\documentclass{siamltex1213}

\usepackage{amsmath,amssymb,amsfonts,cmmib57,mathdots,xcolor}
\usepackage{blkarray}
\usepackage[latin1]{inputenc}
\usepackage{enumerate}

\newcommand{\rev}{\ensuremath\mathrm{rev}}

\newcommand{\FF}{{\mathbb F}}

\newtheorem{remark}{Remark}
\newtheorem{example}{Example}

\title{Block minimal bases $\ell$-ifications of matrix polynomials}
\author{Froil\'an M.\ Dopico\footnotemark[1], Javier P\'{e}rez\footnotemark[2], Paul Van Dooren\footnotemark[3]}
\begin{document}
\maketitle
\slugger{simax}{xxxx}{xx}{x}{x--x}
\renewcommand{\thefootnote}{\fnsymbol{footnote}}
\footnotetext[1]{Departamento de Matem\'{a}ticas, Universidad Carlos III de Madrid,  Avenida de la Universidad 30, 28911, Legan\'{e}s, Spain. Email: {\tt dopico@math.uc3m.es}.
Supported by ``Ministerio de Econom\'{i}a, Industria y Competitividad of Spain" and ``Fondo Europeo de Desarrollo Regional (FEDER) of EU" through grants MTM-2015-68805-REDT, MTM-2015-65798-P (MINECO/FEDER, UE).}
\footnotetext[2]{Department of Mathematical Sciences, University of Montana, USA. Email: {\tt javier.perez-alvaro@mso.umt.edu}.
Partially supported by KU Leuven Research Council grant OT/14/074 and by the Belgian network DYSCO (Dynamical Systems, Control, and Optimization), funded by the Interuniversity Attraction Poles Programme
initiated by the Belgian Science Policy Office.}
\footnotetext[3]{Department of Mathematical Engineering, Universit\'e catholique de Louvain, Avenue Georges
Lema\^itre 4, B-1348 Louvain-la-Neuve, Belgium. Email: {\tt paul.vandooren@uclouvain.be}. Supported by  the Belgian network DYSCO (Dynamical Systems, Control, and Optimization), funded by the Interuniversity Attraction Poles Programme initiated by the Belgian Science Policy Office.
This work was partially developed while Paul Van Dooren held a ``Chair of Excellence'' at Universidad Carlos III de Madrid in the academic year 2017--18.}

\begin{abstract}
The standard way of solving a polynomial eigenvalue problem associated with a matrix polynomial starts by embedding the matrix coefficients of the polynomial into a matrix pencil, known as a strong linearization.
This process transforms the problem into an equivalent generalized eigenvalue problem.
However, there are some situations in which is more convenient to replace linearizations  by other low degree matrix polynomials.
This has  motivated the idea of a strong $\ell$-ification of a matrix polynomial, which is a matrix polynomial of degree $\ell$ having the same finite and infinite elementary divisors, and the same numbers of left and right minimal indices as the original matrix polynomial.
We present in this work a novel method for constructing strong $\ell$-ifications of matrix polynomials of size $m\times n$ and grade $d$ when $\ell< d$, and $\ell$ divides $nd$ or $md$.
This method is based on a family called  ``strong block minimal bases matrix polynomials'', and relies heavily on properties of dual minimal bases.
We show how strong block minimal bases $\ell$-ifications can be constructed from the coefficients of a given matrix polynomial $P(\lambda)$.
We also show that these $\ell$-ifications satisfy many desirable properties for numerical applications: they are strong $\ell$-ifications regardless of whether $P(\lambda)$ is regular or singular, the minimal indices of the $\ell$-ifications   are related to those of $P(\lambda)$ via constant uniform shifts, and eigenvectors and minimal bases of $P(\lambda)$ can be recovered from those of any of the strong block minimal bases $\ell$-ifications.
In the special case where $\ell$ divides $d$, we introduce a subfamily of strong block minimal bases matrix polynomials named ``block Kronecker matrix polynomials'', which is shown to be a fruitful source of companion $\ell$-ifications.
\end{abstract}

\pagestyle{myheadings}
\thispagestyle{plain}
\markboth{F.\ M.\ DOPICO, J. P\'{E}REZ, P.\ VAN DOOREN}{BLOCK MINIMAL BASES $\ell$-IFICATIONS}

\section{Introduction}\label{sec:intro}

Minimal (polynomial) bases are an important type of bases of rational vector subspaces  used extensively in many areas of applied mathematics.
They were introduced by Dedekind and Weber in \cite{Dedekind}, where they are called ``normal bases'', in the context of valuation theory.
Since then, they have played an important role in multivariable linear systems theory, coding theory, control theory, and in the spectral theory of rational and polynomial matrices.
For detailed introductions to minimal bases, their algebraic properties,  computational schemes for constructing such bases from  arbitrary polynomial bases, their robustness under perturbations, and their role in the singular structure of singular rational and polynomial matrices, we refer the reader to the classical works \cite{Gantmacher,Kailath,Wolovich}, the  works \cite{AVV05,Forney,perturbation} and  \cite{Filtrations}, where  an elegant approach to minimal bases via filtrations is presented.


In this paper, we are interested in the use of minimal bases as a tool for solving polynomial eigenvalue problems \cite{rational,FFP,FPJP}. 
We recall that the \emph{complete polynomial eigenvalue problem (CPEP)} associated with a regular matrix polynomial consists in computing all the eigenvalues (finite and infinite) of the  polynomial, while for a singular matrix polynomial, it consists in computing all the eigenvalues (finite and infinite) and all the minimal indices of the polynomial.
One of the most common strategies for solving a CPEP is to transform it into a \emph{generalized eigenvalue problem (GEP)} by using a strong linearization \cite{singular_lin,Strong_lin,Lancaster_book}. 
More specifically, \emph{strong linearization} of a matrix polynomial $P(\lambda)$ is a matrix pencil (i.e., a matrix polynomial of degree at most 1) having the same finite and infinite elementary divisors (and, thus, the same eigenvalues) and the same number of left and right minimal indices. 
Given a strong linearization $L(\lambda)$ of a matrix polynomial $P(\lambda)$, the CPEP associated with $P(\lambda)$ can be solved by applying the QZ algorithm \cite{QZ} or the staircase algorithm \cite{staircase,VanDooren83} to $L(\lambda)$, provided that the minimal indices of $L(\lambda)$ are related with those of $P(\lambda)$ by known rules.
In many applications, the eigenvectors of regular matrix polynomials and the minimal bases of singular matrix polynomials are of interest as well.
In these cases, the strong linearization $L(\lambda)$ should also allow the recovery of eigenvectors and minimal bases of the original matrix polynomial $P(\lambda)$.

Due to their many favorable properties, the most common strong linearizations used in practice to solve CPEP's are the well-known \emph{Frobenius companion forms} \cite{Lancaster_book}.
Indeed, they are constructed from the coefficients of the matrix polynomials without performing any arithmetic operations, they are strong linearizations regardless of whether the matrix polynomials are regular or singular \cite{DTDM10,DTDM12}, the minimal indices of singular polynomials are related with the minimal indices of the Frobenius companion forms by uniform shifts \cite{DTDM10}, the eigenvectors of regular matrix polynomials and minimal bases of singular matrix polynomials  are easily recovered from those of the Frobenius companion forms \cite{DTDM10}, and solving CPEP's by applying a backward stable eigensolver to the Frobenius companion forms is backward stable \cite{FPJP,VanDooren83}.
However, the Frobenius companion forms present some significant drawbacks. 
For instance, they do not preserve any of the most important algebraic structures appearing in applications \cite{GoodVibrations}, they increase significantly the size of the problem, they modify the conditioning of the problem \cite{conditioning}, and they are easily constructed from the  coefficients of the matrix polynomial only if these coefficients are given with respect to the monomial basis.

In the past few years, much effort has been made to constructing strong linearizations that do not present the drawbacks of the Frobenius companion forms.
Concerning the preservation of algebraic structures, there are two main sources available of structure-preserving strong linearizations for structured matrix polynomials. 
The first source is the vector space $\mathbb{DL}(P)$.
This vector space was introduced in \cite{4m-vspace} and further analyzed in \cite{conditioning,BackErrors,GoodVibrations}.
 The second source is based on Fiedler pencils \cite{AV04,DTDM10,DTDM12,Fiedler03} and their different extensions \cite{FPR1,GFPR,FPR2,FPR3,DTDM11,ant-vol11}.
Regarding matrix polynomials expressed in different polynomial bases, strong linearizations can be found in \cite{ACL09,LP16,ChebyFiedler} for the Chebyshev polynomial basis, in \cite{ACL09,Philip} for orthogonal polynomial bases, in \cite{NNT} for degree-graded polynomial bases, in \cite{Bernstein} for the Bernstein polynomial basis, in \cite{Lagrange,Meerbergen} for the Lagrange interpolants basis, and in \cite{Meerbergen} for the Hermite interpolants basis, to name a few recent references.

The notion of strong linearization has been extended to matrix polynomials of arbitrary degree $\ell$ \cite{IndexSum}.
This new notion is named \emph{strong $\ell$-ification}.
A strong $\ell$-ification of a matrix polynomial $P(\lambda)$ is  a matrix polynomial of degree $\ell$ having the same finite and infinite elementary divisors, and the same number of left and right minimal indices as the matrix polynomial $P(\lambda)$.
The first examples of strong $\ell$-ifications of an arbitrary $m\times n$ matrix polynomial $P(\lambda)$ of grade $d$ were given in \cite{IndexSum}.
These $\ell$-ifications were named \emph{Frobenius-like companion forms} of degree $\ell$, because of there resemblance to the Frobenius companion forms.
However, they are defined only when $\ell$ divides $d$.
A more general construction was presented in \cite{FFP}.
This construction is valid for the case where $\ell$ divides $nd$ or $md$, which is the more general condition for which a given construction can provide strong $\ell$-ifications for all matrix polynomials with such size and grade \cite{FFP2}.
Another approach for constructing $\ell$-ifications can be found in \cite{Robol}.
The interest of strong $\ell$-ifications for solving CPEP's  mainly stems from the fact that some even-grade structured matrix polynomials do not have any strong linearization with the same structure due to some spectral obstructions \cite{IndexSum}.
This phenomenon implies the impossibility of constructing structured companion forms for even grades, which suggests that, for even-grade structured matrix polynomials, linearizations should  be replaced by other low-degree matrix polynomials in numerical computations \cite{pal_quadratization}.

There are many (in fact, infinitely many) choices available in the literature for constructing strong linearizations and strong $\ell$-ifications of matrix polynomials.
From a numerical analyst point of view, this situation is very desirable, since one can choose the most favorable construction in terms of various criteria, such as conditioning and backward errors \cite{conditioning,BackErrors}, the basis in which the polynomial is represented \cite{ACL09}, preservation of algebraic structures \cite{pal_quadratization,GoodVibrations}, exploitation of matrix structures in numerical algoritghms \cite{TS-CORK,CORK,Leo2}, etc
However, there has not been a framework providing a way to construct and  analyze all these strong linearizations and strong $\ell$-ifications  in a consistent manner.
Providing such a framework is one of the main goals of this work.
More specifically, our main contributions are the following.
First, we  provide a framework broad enough to accommodate most of the recent work on strong $\ell$-ifications (strong linearizations, strong quadratifications, etc).
This is achieved by introducing the families of \emph{block minimal bases matrix polynomials and strong block minimal bases matrix polynomials}, which unifies the constructions in \cite{FFP2,FPJP,Leo}.
These families rely heavily on the concept of minimal bases \cite{Forney}.
Second, we show that our new framework allows for the construction of infinitely many new strong $\ell$-ifications of matrix polynomials (regular or singular).
These constructions are possible for matrix polynomials of size $m\times n$ and grade $d$ in the case where $\ell$ divides $nd$ or $md$.
In the special case where $\ell$ divides $d$, we introduce the family of \emph{block Kronecker matrix polynomials}.
The advantage of this family over general strong block minimal bases matrix polynomials is that they allow constructing strong $\ell$-ifications without performing any arithmetic operation.
Moreover, some of these $\ell$-ifications are shown to be companion forms \cite[Def. 5.1]{IndexSum} different from the Frobenius-like companion forms in \cite{IndexSum}.
Third,  we  provide the theoretical tools to analyze the algebraic and analytical properties of $\ell$-ifications based on strong block minimal bases pencils in a unified way.
To be more specific, we show how eigenvectors, minimal bases and minimal indices of the matrix polynomials are related with those of the $\ell$-ifications, and that these $\ell$-ifications present one-sided factorizations as those used in \cite{Framework}, useful for performing residual ``local'', i.e., for each particular computed eigenpair, backward error analyses of regular CPEP's solved by using $\ell$-ifications.

We mention in passing that a potential advantage --from the numerical point of view-- of the $\ell$-ifications introduced in this work is that the minimal bases involved in their construction are  particular instances of the so called \emph{full-Sylvester-rank minimal bases} \cite{perturbation}.
Full-Sylvester-rank minimal bases are the only kind of minimal bases for which it is possible to perform a perturbation analysis \cite{perturbation} when fixing the grade of the minimal basis. 
These perturbation results have been successfully used for performing a rigorous backward error analysis of solving CPEP's by using block Kronecker linearizations \cite{FPJP} or by using the $\ell$-ifications introduced by De Ter\'an, Dopico and Van Dooren \cite{FFP,perturbation}.
Thus, we expect that our framework for constructing $\ell$-ifications from dual minimal bases, together with the perturbation results for full-Sylvester-rank minimal bases \cite{perturbation}, will allow us to extend in a future work the backward error results in \cite{perturbation} for the $\ell$-ifications in  \cite{FFP} to a wider family of $\ell$-ifications.

The rest of the paper is as follows.
In Section \ref{sec:def}, we introduce the definitions and notation used throughout the paper, and some basic results needed in other sections.
In Section \ref{sec:minimal_bases}, we recall the notions of minimal bases and dual minimal bases of rational vector subspaces, review some known results for dual minimal bases, and present some basic new results needed in other sections.
In Section \ref{sec:min_ell-ifications}, we introduce the family of strong block minimal bases matrix polynomials, study their properties, and establish the connection of previous works with our work.
 Section \ref{sec:constructing} is devoted to explain the general construction of strong $\ell$-ifications of given matrix polynomials from strong block minimal bases matrix polynomials.
We also introduce in this section the family of block Kronecker matrix polynomials, an important subfamily of  strong block minimal bases matrix polynomials.
This family provides many examples of $\ell$-ifications for matrix polynomials that are constructed without performing any arithmetic operations.
In Section \ref{sec:recovery}, we include  recovery procedures of eigenvectors, and minimal bases and minimal indices of a matrix polynomial from those of its strong block minimal bases matrix polynomials and block Kronecker matrix polynomials.
We also show that strong block minimal bases matrix polynomials admit one-sided factorizations.

\section{Definitions, notation, and some auxiliary results}\label{sec:def}

In this section, we introduce the notation used in the paper, recall some basic definitions and review some basic results.

Although the most relevant case in numerical applications is to consider matrix polynomials with real or complex coefficients, the results in this paper are presented for matrix polynomials with coefficients in arbitrary fields.
Hence, throughout the paper, we use $\mathbb{F}$ to denote an arbitrary field, and $\overline{\mathbb{F}}$ to denote the algebraic closure of $\mathbb{F}$.
By $\mathbb{F}[\lambda]$ and $\mathbb{F}(\lambda)$ we denote, respectively, the ring of polynomials with coefficients from the field $\mathbb{F}$ and the field of rational functions over $\mathbb{F}$.
The set of $m\times n$ matrices with entries in $\mathbb{F}[\lambda]$ is denoted by $\mathbb{F}[\lambda]^{m\times n}$. 
Any element of $\mathbb{F}[\lambda]^{m\times n}$ is called an \emph{$m\times n$ matrix polynomial}, or, just a matrix polynomial.
When $m=1$ (resp. $n=1$), we refer to the matrix polynomial as a \emph{row vector polynomial} (resp. \emph{column vector polynomial}).
A matrix polynomial $P(\lambda)$ is said to be \emph{regular} if it is square and the scalar polynomial $\det P(\lambda)$ is not identically equal to the zero polynomial.
Otherwise, $P(\lambda)$ is said to be \emph{singular}. 
If $P(\lambda)$ is regular and $\det P(\lambda)\in\mathbb{F}$, then $P(\lambda)$ is said to be \emph{unimodular}.
The \emph{normal rank} of a matrix polynomial $P(\lambda)$ is  the rank of $P(\lambda)$ considered as a matrix over the field $\mathbb{F}(\lambda)$.

A matrix polynomial $P(\lambda)\in\mathbb{F}[\lambda]^{m\times n}$ is said to have \emph{grade} $d$ if it can be expressed
in the form
\begin{equation}\label{eq:poly}
P(\lambda)=\sum_{i=0}^d P_i \lambda^i, \quad \mbox{with} \quad P_0,\hdots,P_d\in \mathbb{F}^{m\times n},
\end{equation}
where any of the coefficients, including $P_d$, can be zero. 
The \emph{degree} of $P(\lambda)$, denoted by $\deg(P(\lambda))$, is the maximum integer $k$ such that $P_k\neq 0$. 
When the grade of $P(\lambda)$ is not explicitly stated, we consider its grade equal to its degree.
A matrix polynomial of grade 1 is called a \emph{matrix pencil} or, simply, a pencil.
The vector space of $m\times n$ matrix polynomials of grade $k$ is denoted by $\mathbb{F}_k[\lambda]^{m\times n}$.

For any $d\geq \deg(P(\lambda))$, the \emph{$d$-reversal matrix polynomial} of $P(\lambda)$ is defined as
\[
\rev_d P(\lambda):= \lambda^d P(\lambda^{-1}).
\]
When $P(\lambda)$ is assumed to have grade $d$, then it is assumed that $\rev_d P(\lambda)$ has also grade $d$.

Two matrix polynomials $P(\lambda)$ and $Q(\lambda)$ are said to be \emph{strictly equivalent} if there exist nonsingular constant matrices $E$ and $F$ such that $EP(\lambda)F=Q(\lambda)$, and are said to be \emph{unimodularly equivalent} if there exist unimodular matrix polynomials $U(\lambda)$ and $V(\lambda)$ such that $U(\lambda)P(\lambda)V(\lambda) = Q(\lambda)$.

The \emph{complete eigenstructure} of a regular matrix polynomial consists of its finite and infinite elementary divisors, and for a singular matrix polynomial it consists of its finite and infinite elementary divisors together with its right and left minimal indices.
For more detailed definitions of the complete eigenstructure of matrix polynomials, we refer the reader to \cite[Section 2]{IndexSum}.
Nevertheless, due to its relevance in this work, the singular eigenstructure of matrix polynomials will be briefly reviewed in Section \ref{sec:minimal_bases} (see Definition \ref{def_singular}).

The standard way of solving CPEP's is by linearization.
The definition of linearizations and strong linearizations of matrix polynomials were introduced in \cite{Strong_lin,Lancaster_book} for regular matrix polynomials, and then extended to the singular case in \cite{singular_lin}. 
In \cite{IndexSum}, the notion of (strong) linearization was extended to matrix polynomials of arbitrary degree, given rise to the concept of (strong) $\ell$-ifications. 
All these concepts are introduced in the following definition.
\begin{definition}\label{def:ell-ification}
A matrix polynomial $L(\lambda)$ of degree $\ell>0$ is said to be a \emph{$\ell$-ification} of a given matrix polynomial $P(\lambda)$ of grade $d$ if for some $s\geq 0$ there exist unimodular matrix polynomials $U(\lambda)$ and $V(\lambda)$ such that
\[
U(\lambda)L(\lambda)V(\lambda) = 
\begin{bmatrix}
I_s & 0 \\ 0 & P(\lambda)
\end{bmatrix}.
\]
If, additionally, the matrix polynomial $\rev_\ell L(\lambda)$ is an $\ell$-ification of $\rev_d P(\lambda)$, then $L(\lambda)$ is said to be a \emph{strong $\ell$-ification} of $P(\lambda)$. 
When $\ell=1$, (strong) $\ell$-ifications are called \emph{(strong) linearizations}.
When $\ell=2$, (strong) $\ell$-ifications are called  \emph{(strong) quadratifications}.  
\end{definition}

Any strong $\ell$-ification $L(\lambda)$ of a matrix polynomial $P(\lambda)$  shares with $P(\lambda)$ the same finite and infinite elementary divisors \cite[Theorem 4.1]{IndexSum}. 
However, Definition \ref{def:ell-ification} only guarantees that the number of left (resp. right) minimal indices of $L(\lambda)$ is equal to the number of left (resp. right) minimal indices of $P(\lambda)$.
Except by these constraints on the numbers, $L(\lambda)$ may have a very different set of right and left minimal indices \cite[Corollary 7.12]{IndexSum}. 
Therefore, in the case of singular matrix polynomials, it is important in practice to consider {\em strong $\ell$-ifications with the additional property} that their minimal indices allow  to recover the minimal indices of the polynomial via some simple rule.
The strong $\ell$-ifications introduced in this work present very simple recovery formulas for the minimal indices of the original matrix polynomial.

To easily recognize $\ell$-ifications in certain situations which are of interest in this work, we present Lemma \ref{lemma:antitriglin}. 
This simple result is a simple generalization of \cite[Lemma 2.14]{FPJP}, so we omit its proof.
\begin{lemma} \label{lemma:antitriglin} Let $P(\lambda)\in\mathbb{F}[\lambda]^{m\times n}$  and let $L(\lambda)$ be a degree-$\ell$ matrix polynomial.
 If there exist two unimodular matrix polynomials $\widetilde{U}(\lambda)$ and $\widetilde{V}(\lambda)$ such that
\begin{equation} \label{eq:lintriagrelation}
\widetilde{U}(\lambda)L(\lambda)\widetilde{V}(\lambda) =
\begin{bmatrix}
Z(\lambda) & X(\lambda) & I_{t} \\
Y(\lambda) & P(\lambda) & 0 \\
I_{s} & 0 & 0
\end{bmatrix},
\end{equation}
for some $s \geq 0$ and $t\geq 0$ and  some matrix polynomials $X(\lambda)$, $Y(\lambda)$, and $Z(\lambda)$, then $L(\lambda)$ is an $\ell$-ification of $P(\lambda)$.
\end{lemma}

\section{Minimal indices, minimal bases and dual minimal bases}\label{sec:minimal_bases}

We review in this section the notions of minimal indices of singular matrix polynomials, minimal polynomial bases of rational vector spaces and dual minimal bases.

Recall that any rational subspace $\mathcal{W}$ has bases consisting entirely of vector polynomials.
 The \emph{order} of a vector polynomial basis of $\mathcal{W}$ is defined as the sum of the degrees of its vectors \cite[Definition 2]{Forney}.
  Among all of the possible polynomial bases of  $\mathcal{W}$, those with least order are called \emph{minimal (polynomial) bases} of $\mathcal{W}$ \cite[Definition 3]{Forney}. 
  In general, there are many minimal bases of $\mathcal{W}$, but the ordered list of degrees of the vector polynomials in any of its minimal bases is always the same. 
This list of degrees is called the list of \emph{minimal indices} of $\mathcal{W}$.

\begin{remark} Most of the minimal bases appearing in this work are arranged as the rows of a matrix. 
Therefore, with a slight abuse of notation, we say throughout the paper that an $m\times n$ matrix polynomial (with $m < n$) is a minimal basis if its rows form a minimal basis of the rational subspace they span.
\end{remark}

To work in practice with minimal bases we need the following definitions.
\begin{definition} 
The \emph{$i$th row degree} of a matrix polynomial $Q(\lambda)$ is the degree of the $i$th row of $Q(\lambda)$.
\end{definition}
\begin{definition} \label{def:rowreduced}
Let $Q(\lambda)\in\mathbb{F}[\lambda]^{m\times n}$ be a matrix polynomial with row degrees $d_1,d_2,\hdots,d_m$. 
The {\em highest row degree coefficient matrix} of $Q(\lambda)$, denoted by $Q_h$, is the $m \times n$ constant matrix whose $j$th row is the coefficient of $\lambda^{d_j}$ in the $j$th row of $Q(\lambda)$, for $j=1,2,\hdots,m$. 
The matrix polynomial $Q(\lambda)$ is called {\em row reduced} if $Q_h$ has full row rank.
\end{definition}

Theorem \ref{thm:minimal_basis} gives a very useful characterization of minimal bases. 
This result was originally proved in \cite[Main Theorem-Part 2, p. 495]{Forney}.
 The version we present below can be found in  \cite[Theorem 2.14]{FFP2}.
\begin{theorem}\label{thm:minimal_basis}
The matrix polynomial $Q(\lambda)\in\mathbb{F}[\lambda]^{m\times n}$ is a minimal basis if and only if $Q(\lambda)$ is row reduced and $Q(\lambda_0)$ has full row rank for all $\lambda_0 \in \overline{\FF}$.
\end{theorem}
\begin{remark} Definition \ref{def:rowreduced} and Theorem \ref{thm:minimal_basis} admit obvious extensions to minimal bases arranged as the columns of matrix polynomials, which are used occasionally in this paper (in particular, in Section \ref{sec:recovery}).
\end{remark}

We recall in Definition \ref{def:dualminimalbases} the concept of \emph{dual minimal bases}.
These bases were introduced in \cite{Forney}, and named ``dual minimal bases'' in \cite{FFSP}.
\begin{definition} \label{def:dualminimalbases}
Two matrix polynomials $L(\lambda)\in\mathbb{F}[\lambda]^{m_1\times n}$ and $N(\lambda)\in\mathbb{F}[\lambda]^{m_2\times n}$ are called \emph{dual minimal bases} if $m_1+m_2 = n$, $L(\lambda)N(\lambda)^T=0$ and $L(\lambda)$ and $N(\lambda)$ are both minimal bases.
\end{definition}

\begin{remark}
As in \cite{FPJP}, we use the expression ``$N(\lambda)$ is a minimal basis dual to $L(\lambda)$'', or vice versa, for referring to matrix polynomials $L(\lambda)$ and $N(\lambda)$ as those in Definition \ref{def:dualminimalbases}.
\end{remark}

Theorem \ref{thm:dual_sum} is an important result on the existence of dual minimal bases with prescribed row degrees.
The proof that the condition in Theorem \ref{thm:dual_sum} is necessary comes back at least to \cite{Forney}.
The sufficiency of the condition has been recently proved in \cite[Thm. 6.1]{FFSP}.
\begin{theorem}\label{thm:dual_sum}
There exists a pair of dual minimal bases $K(\lambda)\in\mathbb{F}[\lambda]^{m_1\times (m_1+m_2)}$ and $N(\lambda)\in\mathbb{F}[\lambda]^{m_2\times (m_1+m_2)}$  with row degrees $(\eta_1,\hdots,\eta_{m_1})$ and $(\epsilon_1,\hdots,\epsilon_{m_2})$, respectively, if and only if
\begin{equation}\label{eq:dual_sum}
\sum_{j=1}^{m_1}\eta_j = \sum_{i=1}^{m_2} \epsilon_i.
\end{equation}
\end{theorem}
\begin{remark}\label{remark:zigzag}
Algorithmic procedures for constructing dual minimal bases as those in   Theorem \ref{thm:dual_sum} satisfying \eqref{eq:dual_sum} are presented in \cite[Theorem 6.1]{FFSP}.
Those procedures are based on the construction of \emph{zigzag} and \emph{dual zigzag} matrices \cite[Definitions 3.1 and 3.21]{FFSP} by using the simple algorithm in \cite[Theorem 5.1]{FFSP}.
\end{remark}

A fruitful source of pairs of dual minimal bases that are relevant in this work are the following two matrix polynomials. 
\begin{equation}\label{eq:Lk}
L_k(\lambda):=\begin{bmatrix}
-1 & \lambda  \\
& -1 & \lambda \\
& & \ddots & \ddots \\
& & & -1 & \lambda  \\
\end{bmatrix}\in\mathbb{F}[\lambda]^{k\times(k+1)},
\end{equation}
and
\begin{equation}
  \label{eq:Lambda}
  \Lambda_k(\lambda)^T :=
\begin{bmatrix}
      \lambda^{k} & \cdots & \lambda & 1
\end{bmatrix} \in \mathbb{F}[\lambda]^{1\times (k+1)},
\end{equation}
where here and throughout the paper we occasionally omit some, or all, of the zero entries of a matrix. 
The matrix polynomials $L_k(\lambda)$ and $\Lambda_k(\lambda)^T$ are dual minimal bases \cite[Example 2.2]{FPJP}.
Lemma \ref{lemma:L-Lamb} shows how to obtain easily other pairs of dual minimal bases.

\begin{lemma} \label{lemma:L-Lamb} 
Let $L_k(\lambda)$ and $\Lambda_k(\lambda)^T$ be the matrix polynomials defined, respectively, in \eqref{eq:Lk} and \eqref{eq:Lambda}.
Then, for any $\ell\in\mathbb{N}$ the following statements hold.
\begin{itemize}
\item[\rm (i)] The matrix polynomials $L_k(\lambda^\ell)$ and $\Lambda_k(\lambda^\ell)^T$ are dual minimal bases.
\item[\rm (ii)] For any $p\in\mathbb{N}$, the matrix polynomials $L_k(\lambda^\ell) \otimes I_p$ and $\Lambda_k(\lambda^\ell)^T \otimes I_p$ are dual minimal bases.
\end{itemize}
\end{lemma}
\begin{proof}
Theorem \ref{thm:minimal_basis} guarantees that $L_k(\lambda^\ell)$ and $\Lambda_k(\lambda^\ell)^T$ are minimal bases for any $\ell\in\mathbb{N}$. 
In addition, from $L_k(\lambda^\ell)\Lambda_k(\lambda^\ell)=0$, we conclude that the matrix polynomials $L_k(\lambda^\ell)$ and $\Lambda_k(\lambda^\ell)^T$ are dual minimal bases.
Therefore, part (a) is true.
Part (b) follows from part (a) and \cite[Corollary 2.4]{FPJP}, together with some basic properties of the Kronecker product.
\end{proof}

Notice that the matrix polynomials $L_k(\lambda^\ell)\otimes I_p$ and $\Lambda_k(\lambda^\ell)^T\otimes I_p$ are dual minimal bases with constant row degrees (equal to $\ell$
in the case of $L_k(\lambda^\ell)\otimes I_p$, and equal to $\ell k$ in the case of $\Lambda_k(\lambda^\ell)^T\otimes I_p$). 
For pairs of dual minimal bases with this property, the following result will be useful.
\begin{theorem}{\rm \cite[Theorem 2.7]{FPJP}} \label{thm:minbasesdegreeseq}
The following statements hold.
\begin{enumerate}
\item[\rm (a)] Let $K(\lambda)$ be a minimal basis whose row degrees are all equal to $j$. 
Then $\rev_j K(\lambda)$ is also a minimal basis whose row degrees are all equal to $j$.
\item[\rm (b)] Let $K(\lambda)$ and $N(\lambda)$ be dual minimal bases. 
If the row degrees of $K(\lambda)$ are all equal to $j$ and the row degrees of $N(\lambda)$ are all equal to $\ell$, then $\rev_j K(\lambda)$ and $\rev_\ell N(\lambda)$ are also dual minimal bases.
\end{enumerate}
\end{theorem}

It is well-known that a matrix polynomial $Q(\lambda)$ having full row rank for all $\lambda_0\in\overline{\mathbb{F}}$ can be completed into a unimodular matrix polynomial (see \cite{Kailath} or \cite{unimodular} for efficient algorithms for computing such completions).
Furthermore, this result can be extended to the following theorem.
\begin{theorem}{\rm \cite[Theorem 2.10]{FPJP}}\label{thm:Embedding}
Let $K(\lambda)\in \mathbb{F}[\lambda]^{m_1 \times n}$ and $N(\lambda)\in \mathbb{F}[\lambda]^{m_2 \times n}$ be matrix polynomials such that $m_1 + m_2 = n$, $K(\lambda_0)$ and $N(\lambda_0)$ have both full row rank for all $\lambda_0\in \overline{\mathbb{F}}$, and $K(\lambda) N(\lambda)^T = 0$.
 Then, there exists a unimodular matrix polynomial $U(\lambda)\in \mathbb{F}[\lambda]^{n \times n}$ such that
\[
U(\lambda) =
\begin{bmatrix}
K(\lambda) \\ \widehat{K}(\lambda)
\end{bmatrix}
\quad \mbox{and} \quad
U(\lambda)^{-1}=
\begin{bmatrix}
\widehat{N}(\lambda)^T & N(\lambda)^T
\end{bmatrix}.
\]
\end{theorem}

As a consequence of Theorem \ref{thm:minimal_basis}, Theorem \ref{thm:Embedding} can be applied to any pair of dual minimal bases.
In the case of the dual minimal bases  $L_k(\lambda^\ell)\otimes I_p$ and $\Lambda_k(\lambda^\ell)^T\otimes I_p$, this embedding into unimodular matrix polynomials is particularly simple, as we show in the following example.
\begin{example}\label{ex:unimodular}
Let $L_k(\lambda)$ and $\Lambda_k(\lambda)$ be the matrix polynomials introduced in \eqref{eq:Lk} and \eqref{eq:Lambda}.
If $e_{k+1}$ is the last column of the $(k+1)\times (k+1)$ identity matrix, then the matrix polynomial
\[
V_k(\lambda) = \begin{bmatrix}
L_k(\lambda^\ell) \\  e_{k+1}^T
\end{bmatrix}
\]
is unimodular, and its inverse is given by
\[
V_k(\lambda)^{-1}=
\left[ \begin{array}{ccccc|c}
-1 & -\lambda^\ell & -\lambda^{2\ell}& \cdots & -\lambda^{(k-1)\ell} & \lambda^{k\ell } \\
& -1 & -\lambda^\ell & \ddots & \vdots & \lambda^{(k-1)\ell}\\
& & -1 & \ddots & -\lambda^{2\ell} & \vdots\\
& & & \ddots & -\lambda^{\ell} & \lambda^{2\ell}\\
& & & & -1 & \lambda^\ell \\
& & & & & 1
\end{array}\right].
\]
Notice that the last column of $V_k(\lambda)^{-1}$ is $\Lambda_k(\lambda^\ell)$.
Hence, the matrix $V_k(\lambda)$ is a particular instance of the matrix $U(\lambda)$ in Theorem \ref{thm:Embedding}.
Furthermore, the matrix $V_k(\lambda)\otimes I_p$ is a particular instance of the embedding $U(\lambda)$ for the dual minimal bases $L_k(\lambda^\ell)\otimes I_p$ and $\Lambda_k(\lambda^\ell)^T\otimes I_p$.
\end{example}

Finally, we review the concepts of minimal bases and minimal indices of singular matrix polynomials.
\begin{definition}\label{def_singular}
If a matrix polynomial $P(\lambda)\in\mathbb{F}[\lambda]^{m\times n}$ is singular, then it has non-trivial left and/or right rational null spaces:
\begin{equation} \label{eq:nullspaces}
\begin{split}
\mathcal{N}_\ell(P) & := \{y(\lambda)^T\in\mathbb{F}(\lambda)^{1\times m} \quad \mbox{such that} \quad y(\lambda)^TP(\lambda) = 0\},\\
\mathcal{N}_r(P) & := \{x(\lambda)\in\mathbb{F}(\lambda)^{n \times 1} \quad \mbox{such that} \quad P(\lambda)x(\lambda) = 0\},
\end{split}
\end{equation}
which are particular instances of rational subspaces.
Then, the \emph{left (resp. right) minimal indices and bases of a matrix polynomial} $P(\lambda)$ are defined as those of the rational subspace $\mathcal{N}_\ell(P)$ (resp. $\mathcal{N}_r(P)$).
\end{definition}
\begin{remark}
Given a pair of dual minimal bases $L(\lambda)$ and $N(\lambda)$, notice that the rows of $N(\lambda)$ are a minimal basis for the subspace $\mathcal{N}_r(L)$ and the row degrees of $N(\lambda)$ are the right minimal indices of $L(\lambda)$, and vice versa.
In other words, each $L(\lambda)$ and $N(\lambda)$ provides
a minimal basis for the right nullspace of the other.
\end{remark}

\section{Block minimal bases  matrix polynomials}\label{sec:min_ell-ifications}

We start by introducing the family of \emph{(strong) block minimal bases matrix polynomials} in Definition \ref{def:block_dual_bases_ell}, which is the most important concept introduced in this work.
\begin{definition} \label{def:block_dual_bases_ell} A matrix polynomial
\begin{equation} \label{eq:minbas_ell}
\mathcal{L}(\lambda) = \begin{bmatrix} M(\lambda) & K_2 (\lambda)^T \\ K_1 (\lambda) & 0\end{bmatrix},
\end{equation}
where $M(\lambda)$ is an arbitrary grade-$\ell$ matrix polynomial, is called a {\em block minimal bases degree-$\ell$ matrix polynomial} if $K_1 (\lambda)$ and $K_2(\lambda)$ are both minimal bases of degree $\ell$.
If, in addition, the row degrees of $K_1 (\lambda)$ are all equal to $\ell$, the row degrees of $K_2 (\lambda)$ are all equal to $\ell$, the row degrees of a minimal basis dual to $K_1 (\lambda)$ are all equal, and the row degrees of a minimal basis dual to $K_2 (\lambda)$ are all equal, then $\mathcal{L}(\lambda)$ is called a {\em strong block minimal bases degree-$\ell$ matrix polynomial}.
\end{definition}
\begin{remark}
We allow in Definition \ref{def:block_dual_bases_ell} the border cases
\[
\left[\begin{array}{c|c}
M(\lambda) & K(\lambda)^T 
\end{array}\right] \quad \mbox{or} \quad 
\left[\begin{array}{c}
M(\lambda) \\ \hline
K(\lambda)
\end{array}\right],
\]
where $K(\lambda)$ is a minimal basis of degree $\ell$.
In those cases, the corresponding (strong) block minimal bases  matrix polynomial are said to be \emph{degenerate}.
\end{remark}

Theorem \ref{thm:ell-ification} shows that any (strong) block minimal bases  matrix polynomial is a (strong) $\ell$-ification of a certain matrix polynomial.
\begin{theorem}\label{thm:ell-ification}
Let $K_1 (\lambda)$ and $N_1 (\lambda)$ be a pair of dual minimal bases, and let $K_2 (\lambda)$ and $N_2 (\lambda)$ be another pair of dual minimal bases. 
Consider the matrix polynomial
\begin{equation} \label{eq:Qpolinminbaslin}
Q(\lambda) := N_2(\lambda) M(\lambda) N_1(\lambda)^T,
\end{equation}
and the block minimal bases degree-$\ell$ matrix polynomial $\mathcal{L}(\lambda)$ in \eqref{eq:minbas_ell}. 
Then:
\begin{enumerate}
\item[\rm (a)] $\mathcal{L}(\lambda)$ is an $\ell$-ification of $Q(\lambda)$.
\item[\rm (b)] If $\mathcal{L}(\lambda)$ is a strong block minimal bases degree-$\ell$ matrix polynomial, then $\mathcal{L}(\lambda)$ is a strong $\ell$-ification of $Q(\lambda)$, when $Q(\lambda)$ is considered as a polynomial with grade $\ell + \deg(N_1 (\lambda)) + \deg(N_2 (\lambda))$.
\end{enumerate}
\end{theorem}
\begin{proof}
\noindent {\bf Proof of part (a)}: By Theorem \ref{thm:Embedding}, there exist  unimodular matrix polynomials such that, for $i=1,2$, 
\begin{equation} \label{eq:twounimodembed}
U_i(\lambda) =
\begin{bmatrix}
K_i(\lambda) \\ \widehat{K}_i(\lambda)
\end{bmatrix}
\quad \mbox{and} \quad
U_i(\lambda)^{-1}=
\begin{bmatrix}
\widehat{N}_i(\lambda)^T & N_i(\lambda)^T
\end{bmatrix}.
\end{equation}
If $m_i$ denotes the number of rows of $K_i(\lambda)$, for $i=1,2$, notice that \eqref{eq:twounimodembed} implies $K_i(\lambda) \widehat{N}_i(\lambda)^T = I_{m_i}$ and $K_i(\lambda)  N_i(\lambda)^T = 0$, as this will be important in the argument.
Then, let us consider the unimodular matrices $U_2(\lambda)^{-T} \oplus I_{m_1}$ and $U_1(\lambda)^{-1} \oplus I_{m_2}$.
By a direct matrix multiplication, we obtain
\begin{align}
& (U_2(\lambda)^{-T} \oplus I_{m_1}) \, \mathcal{L}(\lambda) \, (U_1(\lambda)^{-1} \oplus I_{m_2})  \nonumber \\
& \phantom{aaaaaaaaaa} =
\begin{bmatrix}
\widehat{N}_2(\lambda) & 0 \\ N_2(\lambda) & 0 \\0 & I_{m_1}
\end{bmatrix} \,
\begin{bmatrix} M(\lambda) & K_2 (\lambda)^T \\ K_1 (\lambda) & 0\end{bmatrix} \,
\begin{bmatrix}
\widehat{N}_1(\lambda)^T & N_1(\lambda)^T & 0 \nonumber \\
0 & 0 & I_{m_2}
\end{bmatrix} \\
&\phantom{aaaaaaaaaa} =  \begin{bmatrix}
Z(\lambda) & X(\lambda) & I_{m_2} \\
Y(\lambda) & Q(\lambda) & 0 \\
I_{m_1} & 0 & 0
\end{bmatrix}, \label{eq:XYZminlin}
\end{align}
where the matrix polynomials $ X(\lambda), Y(\lambda)$, and  $Z(\lambda)$ are not relevant in this proof. 
Finally, from \eqref{eq:XYZminlin},  Lemma \ref{lemma:antitriglin} proves that $\mathcal{L}(\lambda)$ is an $\ell$-ification of $Q(\lambda)$.

\smallskip

\noindent {\bf Proof of part (b)}: Set $\ell_1:= \deg(N_1 (\lambda))$ and $\ell_2 := \deg(N_2 (\lambda))$. 
Part (b) in Theorem \ref{thm:minbasesdegreeseq}  guarantees that $\rev_\ell K_i(\lambda)$ and $\rev_{\ell_i} N_i(\lambda)$ are dual minimal bases, for $i=1,2$. 
Therefore, the matrix polynomial 
\[
\rev_\ell \mathcal{L}(\lambda) =\left[\begin{matrix}\rev_\ell M(\lambda) & \rev_\ell K_2 (\lambda)^T \\ \rev_\ell K_1 (\lambda) & 0\end{matrix}\right]
\]
 is also a block minimal bases degree-$\ell$ matrix polynomial.
Thus, part (a)  implies that $\rev_\ell \mathcal{L}(\lambda)$ is an $\ell$-ification of
\begin{align*}
(\rev_{\ell_2} N_2(\lambda)) \, (\rev_\ell M(\lambda)) \, (\rev_{\ell_1} N_1(\lambda))^T
& =
\lambda^{\ell_2} N_2 \left( \lambda^{-1} \right) \, \lambda^\ell  \, M \left( \lambda^{-1} \right) \, \lambda^{\ell_1} N_1 \left( \lambda^{-1} \right)^T \\
&=\lambda^{\ell+\ell_1 + \ell_2} Q(\lambda^{-1}) = \rev_{\ell+\ell_1 + \ell_2} Q(\lambda).
\end{align*}
This proves part (b).
\end{proof}

\begin{remark}\label{remark:sizes}
Throughout the rest of the paper,  the sizes of $K_1(\lambda)$ and $K_2(\lambda)$ in Definition \ref{def:block_dual_bases_ell} are denoted without loss of generality by $m_1\times (n+m_1)$ and $m_2\times(m+m_2)$,  respectively.
In other words, we have
\[
\begin{blockarray}{ccc} 
 \begin{block}{[cc]c}
 M(\lambda) & K_2(\lambda)^T  & m+m_2 \\  
 K_1(\lambda) & 0 & m_1 \\ 
\end{block}
  n+m_1 & m_2 & \\  
\end{blockarray}.
\]
With this convention, notice that the sizes of $N_1(\lambda)$ and $N_2(\lambda)$ in Theorem \ref{thm:ell-ification} are $n\times(m_1+n)$ and $m\times (m_2+m)$, respectively, and, thus, $Q(\lambda)$ in \eqref{eq:Qpolinminbaslin} is an  $m\times n$ matrix polynomial.
\end{remark}

 Theorem \ref{thm:ell-ification} shows that every strong block minimal bases matrix polynomial is always a strong $\ell$-ification of a certain matrix polynomial.
In Section \ref{sec:constructing}, we address the inverse problem, that is,  we show how to construct strong $\ell$-ifications for a given matrix polynomial $P(\lambda)$ from strong block minimal bases   matrix polynomials.
But before addressing this important problem, we show in the following section how  previous works on linearizations, quadratifications, and $\ell$-ifications are related with the block minimal bases framework introduced in this section.

\subsection{Previous works related with the block minimal bases matrix polynomials framework}\label{sec:example}

Most of the linearizations, quadratifications and, in general, $\ell$-ifications introduced in previous works fit in the framework of block minimal bases matrix polynomials (modulo some simple operations).
We review some important examples in this section.

\begin{itemize}
\item[\rm(i)] {\bf The Frobenius companion linearizations.}
Let $P(\lambda)=\sum_{i=0}^d P_i\lambda^i\in\mathbb{F}[\lambda]^{m\times n}$.
The most well-known strong linearizations of $P(\lambda)$ are the so called Frobenius companion linearizations \cite{IndexSum,Lancaster_book}
\[
C_1(\lambda) = \left[\begin{array}{cccc}
\lambda P_d+P_{d-1} & P_{d-2} & \cdots &  P_0 \\ \hline
-I_n & \lambda I_n \\
& \ddots & \ddots \\
& & -I_n & \lambda I_n \\
\end{array}\right] = 
\left[ \begin{array}{c} M_1(\lambda) \\ \hline L_{d-1}(\lambda)\otimes I_n \end{array}\right]
\]
and
\[
C_2(\lambda) = \left[\begin{array}{c|ccc}
\lambda P_d+P_{d-1} & -I_m \\
P_{d-2} & \lambda I_m & \ddots  \\
\vdots & & \ddots & -I_m \\
P_0 & &   & \lambda I_m
\end{array}\right]= 
\left[ \begin{array}{c|c} M_2(\lambda) & L_{d-1}(\lambda)^T\otimes I_m\end{array}\right],
\]
where the matrix polynomial $L_k(\lambda)$ has been defined in \eqref{eq:Lk}.
The Frobenius companion forms are degenerate strong block minimal bases pencils.
Moreover, from Theorem \ref{thm:ell-ification} and Lemma \ref{lemma:L-Lamb}, they are strong linearizations of
\begin{align*}
\begin{bmatrix}
\lambda P_d+P_{d-1} & P_{d-2} & \cdots &  P_0
\end{bmatrix}&(\Lambda_{d-1}(\lambda)\otimes I_n)= \\
&(\Lambda_{d-1}(\lambda)^T\otimes I_m)
\begin{bmatrix}
\lambda P_d+P_{d-1}\\P_{d-2} \\ \vdots \\ P_0
\end{bmatrix}=P(\lambda),
\end{align*}
as it is well-known \cite[Theorems 5.3 and 5.4]{IndexSum}.

\item[\rm (ii)] {\bf (Strong) block minimal bases pencils}. 
The family of (strong) block minimal bases pencils introduced in \cite{FPJP} consists in (strong) block minimal bases degree-$\ell$ matrix polynomials with $\ell=1$.
Some important examples in this family are the block Kronecker pencils \cite{FPJP}, the Chebyshev pencils \cite{LP16}, the extended block Kronecker pencils \cite{unified}, the linearization for product bases in \cite{Leo}, and the pencils in block-Kronecker ansatz spaces \cite{Philip_Kronecker}.

\item[\rm(iii)]{\bf Fiedler and Fiedler-like pencils.} 
\emph{Fiedler pencils} were introduced in \cite{Fiedler03} for monic scalar polynomials ($m=n=1$), and then generalized to regular matrix polynomials in \cite{AV04}, to square singular matrix polynomials in \cite{DTDM10}, and to rectangular matrix polynomials in \cite{DTDM12}.
With the goal of constructing large families of structure-preserving linearizations, the families of \emph{generalized Fiedler pencils}, \emph{Fiedler pencils with repetition} and \emph{generalized Fiedler pencils with repetition} \cite{AV04,GFPR,ant-vol11} were introduced.
Very recently, it has been shown in \cite{unified,FPJP} that Fiedler pencils and generalized Fiedler pencils, and, under some generic nonsingularity conditions, Fiedler pencils with repetition and generalized Fiedler pencils with repetition are, modulo permutation, strong block minimal bases pencils.

\item[\rm(iv)] {\bf The standard basis of $\mathbb{DL}(P)$.}
Two vector spaces, denoted by $\mathbb{L}_1(P)$ and $\mathbb{L}_2(P)$ of potential linearizations were introduced in \cite{4m-vspace}.
The intersection of these vector spaces, denoted by $\mathbb{DL}(P)$, was shown to be a fertile source of structure-preserving linearizations \cite{GoodVibrations}.
Since the pencils in the standard basis of the vector space $\mathbb{DL}(P)$ consist  of Fiedler pencils with repetition \cite{FPR1,ant-vol11}, up to permutation and under some generic nonsingularity conditions, they are strong block minimal bases pencils.

\item[\rm(v)]{\bf Linearizations for degree-graded polynomial bases by Amiraslani, Corless and Lancaster.}
In \cite{ACL09}, the authors consider matrix polynomials of the form
\begin{equation}\label{eq_poly_dg}
P(\lambda)=\sum_{i=0}^d P_i\phi_i(\lambda)\in\mathbb{F}[\lambda]^{n\times n},
\end{equation}
where $\{\phi_i(\lambda)\}_{i=0}^\infty$ is a set of degree-graded polynomials satisfying a three-term recurrence relation
\[
\lambda \phi_i(\lambda)=\alpha_i\phi_{i+1}(\lambda)+\beta\phi_i(\lambda)+\gamma\phi_{i-1}(\lambda), \quad i=1,2,\hdots,
\]
where $\alpha_i,\beta_i,\gamma_i$ are real, $\phi_{-1}(\lambda)=0$, $\phi_0(\lambda)=1$, and, if $\kappa_i$ denotes the leading coefficient of $\phi_i(\lambda)$, $0\neq \alpha_i = \kappa_i/\kappa_{i-1}$.
A linearization for \eqref{eq_poly_dg} is given by $\lambda B_\phi-A_\phi$ with\footnote{the linearization in \cite{ACL09} is a permutation of the one presented here.}
\[
B_\phi=\begin{bmatrix}
\kappa_d P_d \\
& I_n \\
& & \ddots \\
& & & I_n
\end{bmatrix}
\]
and
\[
A_\phi=\begin{bmatrix}
-\kappa_{d-1}P_{d-1}+\kappa_d\beta_{d-1}P_d & \alpha_{d-2}I_n & & & \\
-\kappa_{d-1}P_{d-2}+\kappa_{d}\gamma_{d-1}P_d & \beta_{d-2}I_n & \alpha_{d-3}I_n \\
-\kappa_{d-1}P_{d-3} & \gamma_{d-2}I_n & \beta_{d-3}I_n & \ddots \\
\vdots & & \ddots & \ddots & \alpha_0I_n\\
-\kappa_{d-1} P_0 & & & \gamma_1 I_n & \beta_0I_n
\end{bmatrix}.
\]

The pencil $\lambda B_\phi-A_\phi$ is known as the \emph{colleague pencil} when $\{\phi_i(\lambda)\}_{i=0}^\infty$ is the set of Chebyshev polynomials, or as the {\emph comrade pencil} when $\{\phi_i(\lambda)\}_{i=0}^\infty$ is a set of orthogonal polynomials other than the Chebyshev polynomials.
We can write $\lambda B_\phi-A_\phi=\left[\begin{array}{c|c} M_\phi(\lambda) & K_\phi(\lambda)^T\otimes I_n \end{array}\right]$, where
\[
M_\phi(\lambda) = \left[\begin{array}{l}
\lambda\kappa_d P_d+\kappa_{d-1}P_{d-1}-\kappa_d\beta_{d-1}P_d \\
\kappa_{d-1}P_{d-2}-\kappa_{d}\gamma_{d-1}P_d \\
\kappa_{d-1}P_{d-3}\\ 
\hspace{0.5cm} \vdots \\
\kappa_{d-1}P_0
\end{array}\right]\in\mathbb{F}[\lambda]^{dn\times n}
\]
and
\[
K_\phi(\lambda)^T\otimes I_n = \begin{bmatrix}
-\alpha_{d-2} & & & \\
 \lambda-\beta_{d-2} & -\alpha_{d-3} \\
 -\gamma_{d-2} & \lambda-\beta_{d-3} & \ddots \\
 & \ddots & \ddots & -\alpha_0\\
 & & -\gamma_1  & \lambda-\beta_0
\end{bmatrix}\otimes I_n\in\mathbb{F}[\lambda]^{dn\times (d-1)n}.
\]
It is not difficult to show that $K_\phi(\lambda)\otimes I_n$ is a minimal basis  with a dual minimal basis given by 
\[
\Phi_d(\lambda)^T\otimes I_n :=
\begin{bmatrix}
\phi_{d-1}(\lambda) & \cdots & \phi_1(\lambda) & \phi_0(\lambda)
\end{bmatrix}\otimes I_n.
\]
Since $K_\phi(\lambda)\otimes I_n$ has all its row degrees equal to 1, and $\Phi_d(\lambda)^T\otimes I_n$ has all its row degrees equal to $d-1$, we conclude that the pencil $\lambda B_\phi-A_\phi$ is a degenerate strong block minimal bases pencil.
Furthermore, from Theorem \ref{thm:ell-ification}, it is a strong linearization of
\[
(\Phi_d(\lambda)^T\otimes I_n) M_\phi(\lambda) = \kappa_{d-1}P(\lambda),
\]
as it was also proved in \cite{ACL09}.

\item[\rm(vi)]{\bf The Frobenius-like companion $\ell$-ifications.}
The first known construction of strong $\ell$-ifications was presented in  \cite{IndexSum} for the case $\ell$ divides $d$.
These strong $\ell$-ifications where named \emph{Frobenius-like companion forms of degree $\ell$}, because of there resemblance to the first and second Frobenius companion linearizations.
Let $P(\lambda)=\sum_{i=0}^d P_i\lambda^i\in\mathbb{F}[\lambda]^{m\times n}$, and assume that $d=k\ell$, for some $k\in\mathbb{N}$. 
The Frobenius-like companion $\ell$-ifications are constructed as follows.
Based on the coefficients of $P(\lambda)$, let us introduce the following grade-$\ell$ matrix polynomials
\begin{align}\label{eq:B_poly}
\begin{split}
&B_1(\lambda) :=  P_\ell \lambda^\ell +P_{\ell-1}\lambda^{\ell-1}+\cdots +  P_1\lambda + P_0, \\
&B_j(\lambda) :=  P_{\ell j}\lambda^\ell+P_{\ell j-1}\lambda^{\ell-1}+\cdots +  P_{\ell(j-1)+1}\lambda,\quad \mbox{for }j=2,\hdots,k.
\end{split}
\end{align}
Then, the \emph{first and second Frobenius-like companion forms of grade $\ell$} associated with $P(\lambda)$ are, respectively,
\[
\left[\begin{array}{ccccc}
B_k(\lambda) & B_{k-1}(\lambda) & B_{k-2}(\lambda) & \cdots & B_1(\lambda) \\ \hline
-I_n & \lambda^\ell I_n \\
& -I_n & \lambda^\ell I_n \\
& & \ddots & \ddots \\
& & & -I_n & \lambda^\ell I_n
\end{array}\right]= 
\left[ \begin{array}{c} M_1^\ell(\lambda) \\ \hline \phantom{\Big{(}} L_{k-1}(\lambda^\ell)\otimes I_n \phantom{\Big{(}} \end{array}\right]
\]
and
\[
\left[\begin{array}{c|cccc}
B_k(\lambda) & -I_m \\
B_{k-1}(\lambda) & \lambda^\ell I_m & -I_m \\
B_{k-2}(\lambda) & & \lambda^\ell I_m & \ddots \\
\vdots & & & \ddots & -I_m \\
B_1(\lambda) & & & & \lambda^\ell I_m
\end{array}\right]= 
\left[ \begin{array}{c|c} M_2^\ell(\lambda) & L_{k-1}(\lambda^\ell)^T\otimes I_m\end{array}\right],
\]
where $L_k(\lambda)$ is the matrix polynomial in \eqref{eq:Lk}.
The above matrix polynomials are degenerate strong block minimal bases degree-$\ell$ matrix polynomials.
Moreover, from Theorem \ref{thm:ell-ification} and Lemma \ref{lemma:L-Lamb}, they are strong linearizations of
\begin{align*}
\begin{bmatrix}
B_k(\lambda) & B_{k-1}(\lambda) & \cdots &  B_{1}(\lambda)
\end{bmatrix}&(\Lambda_{k-1}(\lambda^\ell)\otimes I_n)= \\
&(\Lambda_{k-1}(\lambda^\ell)^T\otimes I_m)
\begin{bmatrix}
B_s(\lambda) \\ B_{k-1}(\lambda) \\ \vdots \\ B_1(\lambda)
\end{bmatrix}=P(\lambda),
\end{align*}
as it was also proved in \cite{IndexSum} using very different techniques.

\item[\rm(vii)]{\bf The $\ell$-ifications by De Ter\'an, Dopico and Van Dooren.} In \cite{FFP}, the authors provided for the first time an algorithm for constructing strong $\ell$-ifications of a given matrix polynomial $P(\lambda)\in\mathbb{F}[\lambda]^{m\times n}$ of grade $d$, when $\ell$ divides $dn$ or $dm$.
The constructed $\ell$-ifications are of the form 
\[
\left[\begin{array}{c}
\widehat{L}(\lambda) \\ \hline \phantom{\Big{(}} \widetilde{L}(\lambda) \phantom{\Big{(}}
\end{array}\right] \quad \mbox{or} \quad 
\left[\begin{array}{c|c} \widehat{L}(\lambda) & \widetilde{L}(\lambda)^T
\end{array}\right],
\]
where $\widetilde{L}(\lambda)\in\mathbb{F}[\lambda]^{\widehat{m}\times (\widehat{n}+n)}$ is a matrix polynomial of grade $\ell$, and $\widehat{L}(\lambda)\in\mathbb{F}[\lambda]^{\widehat{n}\times (\widehat{n}+n)}$ is a minimal basis with degree $\ell$.
Notice that
\[
\begin{bmatrix}
0 & I_{\widehat{n}} \\ I_{\widehat{m}} & 0 
\end{bmatrix}\left[\begin{array}{c}
\widehat{L}(\lambda) \\ \hline \phantom{\Big{(}} \widetilde{L}(\lambda) \phantom{\Big{(}}
\end{array}\right] =
\left[\begin{array}{c}
\widetilde{L}(\lambda) \\ \hline \phantom{\Big{(}} \widehat{L}(\lambda) \phantom{\Big{(}}
\end{array}\right]
\]
is a degenerate block minimal bases degree-$\ell$  matrix polynomial.
Thus, up to a simple block-permutation, the $\ell$-ifications in \cite{FFP} are block minimal bases matrix polynomials.

\item[\rm(viii)] {\bf The palindromic quadratifications by Huang, Lin, and Su.}
With the aim of devising a structure-preserving algorithm for palindromic matrix polynomials of even grade, in \cite{pal_quadratization}, the authors constructed palindromic strong quadratifications of palindromic matrix polynomials of even grade.
The form of these strong quadratifications depends on whether the grade of $P(\lambda)\in\mathbb{F}[\lambda]^{n\times n}$ is of the form $4s$ or $4s+2$, for some $s\in\mathbb{N}$.
For example, for $d=8$ the quadratification is given by
\[
Q_1(\lambda)=\begin{bmatrix}
\lambda^2 P_5+\lambda P_4+P_3 -\lambda(I_n+P_8P_0) & \lambda^2 P_8 & \lambda^2 P_7+\lambda P_6 & I_n \\
P_0 & -\lambda I_n & \lambda^2 I_n & 0 \\
\lambda P_2+P_1 & I_n & 0 & -\lambda^2 I_n \\
\lambda^2 I_n &0 & -I_n & 0
\end{bmatrix},
\]
and for $d=10$, it is given by
\[
Q_2(\lambda)=\begin{bmatrix}
\lambda^2P_6+\lambda P_5+P_4 & \lambda^2 P_{10}+\lambda P_9 & 0 & \lambda^2 P_8+\lambda P_7 & I_n \\
\lambda P_1+P_0 & 0 & -\lambda^2 I_n & 0 & 0 \\
0 & -I_n & 0 & \lambda^2 I_n & 0 \\
\lambda P_3+P_2 & 0 & I_n & 0 & -\lambda^2 I_n \\
\lambda^2 I_n & 0 & 0 & -I_n & 0
\end{bmatrix}.
\]
It is not difficult to show that there exist two permutation matrices, denoted by $\Pi_1$ and $\Pi_2$, such that $\Pi_1^TQ_1(\lambda)\Pi_1=$
\[
\left[\begin{array}{ccc|c}
\lambda^2 P_5+\lambda P_4+P_3 -\lambda(I_n+P_8P_0) & \lambda^2 P_7+\lambda P_6 & \lambda^2 P_8 & I_n \\
\lambda P_2+P_1 & 0 & I_n & -\lambda^2 I_n \\
P_0 & \lambda_2 I_n & -\lambda I_n & 0 \\ \hline
\phantom{\Big{(}} \lambda^2 I_n \phantom{\Big{(}} & -I_n & 0 & 0
\end{array}\right]
\]
and $\Pi_2^TQ_2(\lambda)\Pi_2=$
\[
\left[\begin{array}{ccc|cc}
\lambda^2P_6+\lambda P_5+P_4 & \lambda^2 P_8+\lambda P_7 & \lambda^2 P_{10}+\lambda P_9 & I_n & 0 \\
\lambda P_3+P_2 & 0 & 0 & -\lambda^2 I_n & I_n \\
\lambda P_1+P_0 & 0 & 0 & 0 & -\lambda^2 I_n \\ \hline
\phantom{\Big{(}} \lambda^2 I_n \phantom{\Big{(}} & -I_n & 0 & 0 & 0 \\
0 & \lambda^2 I_n & -I_n & 0  & 0
\end{array}\right],
\]
which are block minimal bases quadratic matrix polynomials.
Furthermore, it is easily checked that $Q_2(\lambda)$ is, in fact, a strong block minimal bases quadratic matrix polynomial.
These results are easily generalized for any even grade.
Hence, the  quadratifications introduced in \cite{pal_quadratization} are, up to permutations, block minimal bases quadratic matrix polynomials.
\end{itemize}

\medskip

The above list is just a sample of linearizations, quadratifications, and $\ell$-ifications given in order to show that a great part of the recent work on $\ell$-ifications (linearizations, quadratifications, etc) is included in the  block minimal bases matrix polynomials framework.
Many other constructions fit also in this framework \cite{ChebyFiedler,Leo,Meerbergen}.

\section{Constructing a strong block minimal bases $\ell$-ification for a given matrix polynomial}\label{sec:constructing}
We return to the problem left open at the end of Section \ref{sec:minimal_bases}.
We show how to construct  strong $\ell$-ifications of a prescribed matrix polynomial $P(\lambda)\in\mathbb{F}[\lambda]^{m\times n}$ of grade $d$ when $\ell$ divides $md$ or $nd$.
We focus on the case $\ell$ divides $md$.
The case $\ell$ divides $nd$ will be considered in Section \ref{sec:ell_divides_nell}.

From Theorem \ref{thm:ell-ification}, we obtain the following procedure for constructing  strong $\ell$-ifications of a given matrix polynomial $P(\lambda)\in\mathbb{F}[\lambda]^{m\times n}$ of grade $d$ from strong block minimal bases matrix polynomials.

\medskip

\begin{itemize}
\item[{\rm \bf Step 1}] Choose two pairs of dual minimal bases $K_1(\lambda)$ and $N_1(\lambda)$, and $K_2(\lambda)$ and $N_2(\lambda)$, with sizes as in Remark \ref{remark:sizes}, such that all the row degrees of $K_1(\lambda)$ and $K_2(\lambda)$ are equal to $\ell$, all the row degrees of $N_1(\lambda)$ are equal to $\epsilon$, and all the row degrees of $N_2(\lambda)$ are equal to $\eta$, with $\epsilon+\eta=d-\ell$.
\item[{\rm \bf Step 2}] Solve the matrix polynomial equation $N_2(\lambda)M(\lambda)N_1(\lambda)^T = P(\lambda)$ for $M(\lambda)$ with grade $\ell$.
\end{itemize}

\medskip

We consider, first, the problem of constructing the pairs of dual minimal bases $K_1(\lambda)$ and $N_1(\lambda)$, and $K_2(\lambda)$ and $N_2(\lambda)$  in {\bf Step 1}.
Then, we show that for each minimal bases $N_1(\lambda)$ and $N_2(\lambda)$ obtained from {\bf Step 1}, the polynomial equation in {\bf Step 2} has infinitely many solutions with grade $\ell$.
When $\deg(P(\lambda))=d$, all such solutions have degree  equal to $\ell$.

\subsection{Solving {\bf Step 1}}

There are some constraints on the degrees and sizes of the dual minimal bases in {\bf Step 1} that follow from Theorem \ref{thm:dual_sum}.
Indeed, we obtain from Theorem \ref{thm:dual_sum} that the two pairs of dual minimal bases in {\bf Step 1} exist if and only if the linear system
\begin{equation}\label{eq:system}
\begin{bmatrix}
\ell & 0 & -n & 0 \\
0 & \ell & 0 & -m \\
0 & 0 & 1 & 1 \\
\end{bmatrix}
\begin{bmatrix}
m_1\\ m_2\\ \epsilon\\ \eta
\end{bmatrix} =
\begin{bmatrix}
0 \\ 0 \\ d-\ell
\end{bmatrix}
\end{equation}
has at least one non-negative integer solution.
When $\ell<d$ and $\ell$ divides  $md$ this is always the case,  since $md = s\ell$, for some non-zero natural number $s$, implies that $m_2 = s-m$, $m_1=0$, $\epsilon =0$ and $\eta= d-\ell$ is a non-negative integer solution of \eqref{eq:system}.
Moreover, as we show in Proposition \ref{prop:all_solutions}, there may exist many more non-negative integer solutions of \eqref{eq:system} under the hypothesis $\ell<d$ and $\ell$ divides  $md$.
\begin{proposition}\label{prop:all_solutions}
Given natural numbers $\ell,d,n,m$  such that $d>\ell$ and $\ell$ divides $md$, set
\[
\gamma:=\frac{\ell}{\gcd\{\ell,n,m\}}.
\]
Then, the vectors
\begin{equation}\label{eq:solutions}
\begin{bmatrix}
\dfrac{kn\gamma}{\ell} \\
\dfrac{md}{\ell}-\frac{km\gamma}{\ell}-m \\
k\gamma \\
d-\ell -k\gamma
\end{bmatrix} \quad \mbox{with} \quad k=0,1,\hdots,\left\lfloor  (d-\ell)/\gamma \right\rfloor,
\end{equation}
are the non-negative integer solutions of \eqref{eq:system}, where $\lfloor \cdot \rfloor$ denotes the floor function.
\end{proposition}
\begin{proof}
The real solutions of \eqref{eq:system} are given by
\begin{equation}\label{eq:vector_aux}
\begin{bmatrix}
\dfrac{n\epsilon}{\ell}\\
\dfrac{md}{\ell}-\frac{m\epsilon}{\ell}-m\\
\epsilon \\
d-\ell-\epsilon
\end{bmatrix} \quad \mbox{with }\epsilon\in\mathbb{R}.
\end{equation}
Thus, the problem is reduced to find the values $\epsilon\in\{0,1,\hdots,d-\ell\}$ for which the vector \eqref{eq:vector_aux} has non-negative integer entries.
Since $\ell$ divides $md$ by assumption, this problem is equivalent to find those values of $\epsilon$ that make $n\epsilon/\ell$ and $m\epsilon/\ell$ non-negative integers.
To finish the proof, it suffices to notice that (1) both $n\epsilon/\ell$ and $m\epsilon/\ell$ are non-negative integers if and only if $\epsilon$ is a multiple of $\gamma$, and (2) the entries of the vectors are all non-negative because $k\leq (d-\ell)/\gamma$.
\end{proof}

Once  some non-negative values for $m_1$, $m_2$, $\epsilon$ and $\eta$ satisfying the linear system \eqref{eq:system} have been fixed,  Theorem \ref{thm:dual_sum} guarantees the existence of the two pairs of dual minimal bases in {\bf Step 1}.
In order to construct those pairs of dual minimal bases, one may consider, as we pointed out in Remark \ref{remark:zigzag}, the procedures in \cite[Theorems 5.1 and 6.1]{FFSP} or \cite[Theorem 5.3]{FFSP} based on zigzag matrices.


\subsection{Solving Step 2}
We now show that the equation
\begin{equation}\label{eq:Step2}
N_2(\lambda)M(\lambda)N_1(\lambda)^T=P(\lambda) 
\end{equation}
is always consistent (with infinitely many solutions) when $N_1(\lambda)$ and $N_2(\lambda)$ are any pair of minimal bases obtained from {\bf Step 1}. 
We assume that both $\epsilon$ and $\eta$ are nonzero, otherwise the consistency of \eqref{eq:Step2} follows from the results in \cite[Section 4.1]{FFP}.
We split {\bf Step 2} into two substeps:

\medskip

\begin{itemize}
\item[{\rm \bf Step 2.1}] Solve the equation $N_2(\lambda)B(\lambda) = P(\lambda)$ for $B(\lambda)$ with grade  $d-\eta$.
\item[{\rm \bf Step 2.2}] Solve the equation $M(\lambda)N_1(\lambda)^T = B(\lambda)$ for $M(\lambda)$ with grade $\ell$.
\end{itemize}

\smallskip

The consistency of the equations in {\bf Steps 2.1 and 2.2} follows from the fact that both $N_1(\lambda)$ and $N_2(\lambda)$  are minimal bases with constant row degrees whose right minimal indices are all equal to $\ell$.
This motivates Lemma \ref{lemma:solving-step2}, where {\em convolution matrices}\footnote{Convolution matrices are called Sylvester matrices in \cite{perturbation}.
More specifically, the convolution matrix $C_j(Q)$ is the Sylvester matrix $S_{j+1}(Q)$, $j=0,1,\hdots$} will be used.
 For any matrix polynomial $Q(\lambda) = \sum_{i=0}^q Q_i \lambda^i$ of grade $q$ and arbitrary size, we define the sequence of its {\em convolution matrices} as follows
\begin{equation} \label{eq:defconvolution}
C_j (Q) :=
\underbrace{\left[
\begin{array}{cccc}
Q_q \\
Q_{q-1} & Q_{q} \\
\vdots &  Q_{q-1} & \ddots \\
Q_0 & \vdots & \ddots & Q_{q} \\
    & Q_0    & \vdots & Q_{q-1} \\
    &        & \ddots &   \vdots  \\
    &        &        & Q_0
\end{array}
\right]}_{\displaystyle j+1 \; \mbox{block columns}} , \quad \mbox{for $j=0,1,2,\ldots$.}
\end{equation}
Notice that for $j=0$, the matrix $C_0 (Q)$ is a block column matrix whose block entries are the matrix coefficients of the polynomial $Q(\lambda)$.

\begin{lemma}\label{lemma:solving-step2}
Let $K(\lambda)\in\mathbb{F}[\lambda]^{s\times (s+t)}$ and $N(\lambda)\in\mathbb{F}[\lambda]^{t\times (s+t)}$ be dual minimal bases such that all the row degrees of $N(\lambda)$ are equal to $n$ and all the row degrees of $K(\lambda)$ are equal to $k$.
Let $Q(\lambda)\in\mathbb{F}[\lambda]^{t\times r}$ be a matrix polynomial of grade $n+b$, with $b\geq k$.
Then, the following statements hold.
\begin{itemize}
\item[\rm (a)] The equation
\begin{equation}\label{eq:NB=Q}
N(\lambda)B(\lambda)=Q(\lambda)
\end{equation}
for the unknown matrix polynomial $B(\lambda)$ of grade $b$ has infinitely many solutions.
Moreover, when $\deg(Q(\lambda))=n+b$, all of such solutions have degree equal to $b$.
\item[\rm (b)] The set of solutions of \eqref{eq:NB=Q} depends on $(b-k+1)sr$ free parameters.
\item[\rm (c)] If $B_0(\lambda)$ is a particular solution of \eqref{eq:NB=Q}, then any other solution of \eqref{eq:NB=Q} can be written as
\[
B(\lambda)=B_0(\lambda)+K(\lambda)^T X(\lambda),
\]
for some $X(\lambda)\in\mathbb{F}_{b-k}[\lambda]^{s\times r}$.
\end{itemize}
\end{lemma}
\begin{proof}
{\bf Proof of part (a)}.
Let us write $N(\lambda) = \sum_{i=0}^n  N_{i}\lambda^i$, $B(\lambda)=\sum_{i=0}^bB_{i}\lambda^{i}$ and $Q(\lambda)=\sum_{i=0}^{n+b}Q_{i}\lambda^{i}$.
Equating matrix coefficients on both sides of \eqref{eq:NB=Q}, we obtain the block-linear system 
\begin{equation}\label{eq:conv-N}
\underbrace{\begin{bmatrix}
N_{n} \\
\vdots & N_{n} \\
N_{0} & \vdots & \ddots \\
 & N_{0} & \vdots & N_{n} \\
 & & \ddots & \vdots \\
 & & & N_{0}
\end{bmatrix}}_{b+1 \mbox{ block columns }}
\begin{bmatrix}
B_{b} \\ B_{b-1} \\ \vdots \\ B_0
\end{bmatrix} =
\begin{bmatrix}
Q_{n+b} \\ Q_{n+b-1} \\ \vdots \\ Q_0
\end{bmatrix},
\end{equation}
or, using convolution matrices, $C_{b}(N)C_0(B)=C_0(Q)$.
We will show that this linear system is consistent by showing that the matrix $C_{b}(N)$ has full row rank.
To do this, let us partition the  matrix $C_{b}(N)$ as follows
\[
C_{b}(N)=\begin{bmatrix}
A_{11}(N) & 0 \\
A_{21}(N) & C_{k-1}(N)
\end{bmatrix},
\]
where $A_{11}(N)$ corresponds to the upper-left $(b+1-k)\times (b+1-k)$ block submatrix of $C_{b}(N)$, which is of the form
\[
A_{11}(N) = 
\begin{bmatrix}
N_{n} \\
* & \ddots \\
* & * & N_{n}
\end{bmatrix},
\]
where ``$*$'' denotes the parts of $A_{11}(N)$ that are not relevant for the argument.
Notice that the matrix $A_{11}(N)$ has full row rank because $N_{n}$ has full row rank.
Thus, one can solve for $B_{b},\hdots,B_{k}$ from  the linear system
\[
A_{11}(N)
\begin{bmatrix}
B_{b} \\ B_{b-1} \\ \vdots \\ B_k
\end{bmatrix}
= \begin{bmatrix}
Q_{n+b} \\ Q_{n+b-1} \\ \vdots \\ Q_{n+k}
\end{bmatrix},
\]
since the above system is always consistent.
Additionally, when $Q(\lambda)$ is assumed to have degree $n+b$, i.e., $Q_{n+b}\neq 0$, notice that we have $N_{n}B_{b}=Q_{n+b}$, which implies $B_{b}\neq 0$ because $N_n$ has full row rank.
Hence, when $\deg(Q(\lambda))=n+b$, all the solutions of \eqref{eq:NB=Q} have exactly degree $b$.

Next, we can solve for $B_{k-1},\hdots,B_0$ from 
\[
C_{k-1}(N)
\begin{bmatrix}
B_{k-1} \\ B_{k-2} \\ \vdots \\ B_0 
\end{bmatrix}=
\begin{bmatrix}
Q_{n+k-1} \\ Q_{n+k-2} \\ \vdots \\ Q_0
\end{bmatrix}-
A_{21}(N)
\begin{bmatrix}
B_{b} \\ B_{b-1} \\ \vdots \\ B_k 
\end{bmatrix},
\]
which has a unique solution since the matrix $C_{k-1}(N)$ is nonsingular.
The nonsingularity of $C_{k-1}(N)$ follows from the following argument.
First, applying Theorem \ref{thm:dual_sum} to the dual minimal bases $K(\lambda)$ and $N(\lambda)$, we obtain $nt=ks$.
Then, notice that the matrix $C_{k-1}(N)$ has size $(n+k)t\times (s+t)k$ or, using $nt=ks$, $(n+k)t\times (n+k)t$.
Hence, it is a square matrix.
Finally, note that $K(\lambda)$ is a full-Sylvester-rank matrix (see, for example, \cite[Theorem 4.4]{perturbation}) and, thus, all its convolution matrices have full rank.
Therefore, the matrix $C_{k-1}(N)$ is nonsingular.

\medskip

\noindent {\bf Proof of part (b)}.
Let us introduce the following linear operator
\begin{align}\label{eq:Phi}
\begin{split}
\Phi_N:\mathbb{F}_{b}[\lambda]^{(s+t)\times r}&\longrightarrow \mathbb{F}_{n+b}[\lambda]^{t\times r}\\
B(\lambda) &\longrightarrow N(\lambda)B(\lambda),
\end{split}
\end{align}
which, by part (a), is surjective.
Since $\mathbb{F}_b[\lambda]^{(s+t)\times r}\cong\mathbb{F}^{(s+t)r(b+1)}$ and $\mathbb{F}_{n+b}[\lambda]^{t\times r}\cong \mathbb{F}^{tr(n+b+1)}$, and using $nt=k s$ (which, we recall, follows from applying Theorem \ref{thm:dual_sum} to $K(\lambda)$ and $N(\lambda)$), we readily obtain that
\[
\mathrm{dim}(\mathrm{null}(\Phi_N))=(b-k+1)sr,
\]
which shows that the set of solutions of \eqref{eq:NB=Q} depends on $(b-k+1)sr$ free parameters.

\medskip

\noindent {\bf Proof of part (c)}. 
Note that the set of matrix polynomials of the form $K(\lambda)^TX(\lambda)$, where $X(\lambda)\in\mathbb{F}[\lambda]^{s\times r}$ is a matrix polynomial of grade $b-k$, is a vector subspace that is contained in $\mathrm{null}(\Phi_N)$, with $(b-k+1)sr$ free parameters (the entries of the matrix coefficients of $X(\lambda)$).
Hence, it suffices to show that the mapping $X(\lambda)\rightarrow K(\lambda)^T X(\lambda)$ is injective. 
Indeed $K(\lambda)^T X(\lambda)=0$ can only hold if $X(\lambda)=0$ because $K(\lambda)$ has full normal row rank.
\end{proof}

As a consequence of Lemma \ref{lemma:solving-step2}, the convolution matrix $C_b(N)$ in \eqref{eq:conv-N} has full row rank.
Hence, the matrix $C_b(N)^\dagger C_0(Q)$ is a solution of \eqref{eq:conv-N}, where $A^\dagger$ denotes the Moore-Penrose pseudoinverse of a matrix $A$.
This motivates the following definition.
\begin{definition}
Let $N(\lambda)\in\mathbb{F}[\lambda]^{t\times (s+t)}$ be a minimal basis with all its row degrees equal to $n$ and with all its right minimal indices equal to $k$, and let $\Phi_N$ be the linear operator in \eqref{eq:Phi}.
Then, we introduce the linear operator
\begin{align}\label{eq:Phi-dagger}
\begin{split}
\Phi_N^\dagger:\mathbb{F}_{n+b}[\lambda]^{t\times r}&\longrightarrow \mathbb{F}_{b}[\lambda]^{(s+t)\times r}\\
Q(\lambda) &\longrightarrow \Phi_N^\dagger[Q](\lambda)=B(\lambda),
\end{split}
\end{align}
where $B(\lambda)$ is defined by partitioning  $C_b(N)^\dagger C_0(Q)$ into $b+1$ blocks of size $(s+t)\times r$ and interpreting these blocks as the matrix coefficients of $B(\lambda)$, i.e., $C_0(B)=C_b(N)^\dagger C_0(Q)$.
The matrix polynomial $\Phi_N^\dagger[Q](\lambda)$ is called \emph{the minimum norm solution to $N(\lambda)B(\lambda)=Q(\lambda)$}.
\end{definition}

We finally show in Theorem \ref{thm:solution_final} that the equation \eqref{eq:Step2} is consistent for every $P(\lambda)$, determine the number of free parameters that its set of solution depends on, and give a concise characterization of this set.
To do this, let us notice that the linear operator 
\begin{align}\label{eq:Psi}
\begin{split}
\Psi_{(N_1,N_2)}:\mathbb{F}_{\ell}[\lambda]^{(m+m_2)\times (n+m_1)}&\longrightarrow \mathbb{F}_{d}[\lambda]^{m\times n}\\
M(\lambda) &\longrightarrow  N_2(\lambda)M(\lambda)N_1(\lambda)^T
\end{split}
\end{align}
can be written as the composition $\Psi_{(N_1,N_2)} = \phi_{N_2} \circ (\cdot)^T \circ \phi_{N_1}\circ (\cdot)^T$, where $(\cdot)^T$ denotes the transpose operation and $g\circ f$ denotes the composition of $g$ with $f$.

\begin{theorem}\label{thm:solution_final}
Let $P(\lambda)\in\mathbb{F}[\lambda]^{m\times n}$ be a matrix polynomial of grade $d$, and let $\ell < d$.
Let $K_1(\lambda)$ and $N_1(\lambda)$, and  $K_2(\lambda)$ and $N_2(\lambda)$ be two pairs of dual minimal bases obtained from {\bf Step 1}.
Then, the following statements hold.
\begin{itemize}
\item[\rm (a)] The equation \eqref{eq:Step2}  for the unknown matrix polynomial $M(\lambda)$ of grade $\ell$ has infinitely many solutions.
When $\deg(P(\lambda))=d$, all of such solutions have degree equal to $\ell$.
\item[\rm (b)] The set of solutions of \eqref{eq:Step2} depends on $m_2n(\epsilon+1)+(m_2+m)m_1$ free parameters.
\item[\rm (c)] Let $\Psi^\dagger_{(N_1,N_2)}:=(\cdot)^T\circ\phi^\dagger_{N_1}\circ (\cdot)^T \circ \phi_{N_2}^\dagger$.
Then, any solution of \eqref{eq:Step2} can be written as
\[
M(\lambda) = \Psi^\dagger_{(N_1,N_2)}[P](\lambda)+\Phi_1^\dagger[X^TK_2](\lambda)^T+Y^TK_1(\lambda),
\]
for some $Y\in\mathbb{F}^{m_1\times (m_2+m)}$ and $X(\lambda)\in\mathbb{F}_\epsilon[\lambda]^{m_2\times n}$.
\end{itemize}
\end{theorem}
\begin{proof}
The results follow by applying repeatedly Lemma \ref{lemma:solving-step2} to $N_2(\lambda)B(\lambda)=P(\lambda)$ and $N_1(\lambda)M(\lambda)^T=B(\lambda)^T$, taking the minimum norm solutions as particular solutions.
\end{proof}

\begin{remark}
Given a particular solution $M_0(\lambda)$ of  \eqref{eq:Step2}, the set of grade-$\ell$ matrix polynomials of the form $M_0(\lambda)+K_2(\lambda)^TX+YK_1(\lambda)$, where $X,Y$ are arbitrary matrices of appropriate size,  is a subset of the set of solutions of \eqref{eq:Step2}.
\end{remark}
\begin{remark}
Given a particular solution $M_0(\lambda)$ of \eqref{eq:Step2}, the set of solutions of \eqref{eq:Step2} takes a simpler form than the one in part (c) in Theorem \ref{thm:solution_final} in three cases.
The first  case is when $\ell=1$.
In this situation, the set of solutions of \eqref{eq:Step2} is equal to the set of pencils of the form
\[
M_0(\lambda)+K_2(\lambda)^TX+YK_1(\lambda),
\]
where $X,Y$ are arbitrary matrices of appropriate sizes.
The other two cases are when either $m_1=0$ or $m_2=0$.
In the former case, the set of solutions of \eqref{eq:Step2} is equal to the set of matrix polynomials of the form $M_0(\lambda)+K_2(\lambda)^TX$, and in the latter, $M_0(\lambda)+YK_1(\lambda)$, where $X,Y$ are again arbitrary matrices of appropriate sizes.
\end{remark}

\smallskip

In Example \ref{ex:sym_quad1}, we apply our new procedure for constructing strong $\ell$-ifications to the problem of quadratizacing a symmetric matrix polynomial in a structure-preserving way.
The interest of this example stems from the fact that there are symmetric matrix polynomials with even grade for which it is impossible to construct symmetric strong linearizations \cite[Theorem 7.20]{IndexSum}.
\begin{example}\label{ex:sym_quad1}
Let $P(\lambda)=\left[ \begin{smallmatrix} \lambda^4 & 0 \\ 0 & 0 \end{smallmatrix} \right]$, which is symmetric, that is, $P(\lambda)^T=P(\lambda)$.
The matrix polynomial $P(\lambda)$ is singular with exactly one right minimal index and one left minimal index, both equal to zero.
Hence, $P(\lambda)$ does not have any symmetric strong linearization by \cite[Corollary 7.19]{IndexSum}.
Nevertheless, we show in this example that $P(\lambda)$ can be ``quadratized'' in a structure-preserving way.
To do this, we use a strong block minimal bases quadratic matrix polynomial of the form
\begin{equation}\label{eq:example-sym}
\left[ \begin{array}{c|c}
M(\lambda) & K(\lambda)^T \\ \hline
K(\lambda) & 0
\end{array}\right]
\end{equation}
where $K(\lambda)$ and a dual minimal basis to $K(\lambda)$, denoted by $N(\lambda)$, are given, respectively, by
\[
K(\lambda) = \begin{bmatrix}
1 & -\lambda & \lambda^2
\end{bmatrix} \quad \mbox{and} \quad 
N(\lambda) = \begin{bmatrix}
\lambda & 1 & 0 \\
0 & \lambda & 1
\end{bmatrix}.
\]
From Theorem \ref{thm:ell-ification}, we see that in order to obtain a symmetric strong quadratification of the form \eqref{eq:example-sym}, we need to solve
\[
\begin{bmatrix}
\lambda^4 & 0 \\ 0  & 0 
\end{bmatrix} =
\begin{bmatrix}
\lambda & 1 & 0 \\
0 & \lambda & 1
\end{bmatrix}
M(\lambda)
\begin{bmatrix}
\lambda & 0 \\ 1 & \lambda \\ 0 & 1
\end{bmatrix}
\]
for a symmetric $M(\lambda)$ with degree equal to 2. 
The quadratic matrix polynomial $M(\lambda)=\left[\begin{smallmatrix} \lambda^2 & 0 & 0 \\ 0 & 0 & 0 \\ 0 & 0 & 0 \end{smallmatrix}\right]$ is one of such solutions.
Therefore, we conclude that
\[
\left[ \begin{array}{ccc|c}
\lambda^2 & 0 & 0 & 1 \\ 0 & 0 & 0 & -\lambda \\ 0 & 0 & 0 & \lambda^2 \\ \hline
1 & -\lambda & \lambda^2 & 0
\end{array}\right]
\]
is a symmetric strong quadratification of $P(\lambda)$.
\end{example}

\subsection{When $\ell$ divides $nd$}\label{sec:ell_divides_nell}

We consider in this subsection the problem of constructing strong $\ell$-ifications using strong block minimal bases matrix polynomials in the case when $\ell$ divides $nd$.
Our construction follows from the following lemma.
\begin{lemma}\label{lemma:transpose}
Let $\mathcal{L}(\lambda)$ be a strong block minimal bases degree-$\ell$ matrix polynomial as in \eqref{eq:minbas_ell}, and let $Q(\lambda)$ be the matrix polynomial in \eqref{eq:Qpolinminbaslin}.
Then, $\mathcal{L}(\lambda)^T$ is also a strong block minimal bases degree-$\ell$ matrix polynomial, which is a strong $\ell$-ification of the matrix polynomial $Q(\lambda)^T$.
\end{lemma}
\begin{proof}
Clearly, $\mathcal{L}(\lambda)^T$ is also a strong block minimal bases degree-$\ell$ matrix polynomial with the roles of $(K_1(\lambda), N_1(\lambda))$ and $(K_2(\lambda),N_2 (\lambda))$ interchanged.
Thus,  $\mathcal{L}(\lambda)^T$ is a strong $\ell$-ification of $Q(\lambda)^T$.
\end{proof}

Let $P(\lambda)\in\mathbb{F}[\lambda]^{m\times n}$ be a matrix polynomial with degree $d$, and assume that there is $\ell$ such that $\ell$ divides $nd$.
We obtain from Lemma \ref{lemma:transpose} that if the process developed in the previous section for the case when $\ell$ divides $md$ is applied to the matrix polynomial $P(\lambda)^T$ (of size $n\times m$ and degree $d$), then a strong $\ell$-ification $\mathcal{L}(\lambda)^T$ of $P(\lambda)^T$ is constructed, and this gives a strong $\ell$-ification $\mathcal{L}(\lambda)$ of $P(\lambda)$.

\subsection{When $\ell$ divides $d$: block Kronecker matrix polynomials and companion $\ell$-ifications}\label{sec:companion}

In applications, the most important type of strong $\ell$-ifications are the  so called \emph{companion $\ell$-ifications} \cite[Definition 5.1]{IndexSum}, also known as companion forms. 
In words, companion $\ell$-ifications are uniform templates for constructing matrix polynomials $L(\lambda)=\sum_{i=0}^\ell L_i\lambda^i$ of degree $\ell$, which are strong $\ell$-ifications for any matrix polynomial $P(\lambda)=\sum_{i=0}^d P_i\lambda^i$ of fixed grade and size. 
Furthermore, for $i=0,\hdots,\ell$, the entries of $L_i$ are scalar-valued function of the entries of $P_0,P_1,\hdots,P_d$.
 These scalar-valued functions are either a constant or a constant multiple of just one of the entries of $P_0,P_1,\hdots,P_d$.
For $\ell>1$, the only known example of companion $\ell$-ifications are the Frobenius-like companion $\ell$-ifications in \cite{IndexSum}.
 For $\ell=1$, many other companion linearizations exist \cite{DTDM10}. 

One of the aims of this section is to present a procedure for constructing new companion $\ell$-ifications.
We start by introducing a subfamily of strong block minimal bases  matrix polynomials.
This family generalizes the block Kronecker pencils  \cite{FPJP} from $\ell=1$ to any degree $\ell$.
The advantage of this family  over general strong block minimal bases matrix polynomials as in \eqref{eq:minbas_ell} is that  it is very easy to characterize the set of $(1,1)$ blocks $M(\lambda)$ that make them strong $\ell$-ifications of a prescribed matrix polynomial $P(\lambda)$.

\begin{definition} \label{def:blockKronlin} Let $L_k (\lambda)$ be the matrix polynomial defined in \eqref{eq:Lk} and let $M(\lambda)$ be an arbitrary matrix polynomial of grade $\ell$.
 Then any matrix polynomial of the form
\begin{equation}
  \label{eq:Kronecker_ell}
  \mathcal{L}(\lambda)=
  \left[
    \begin{array}{c|c}
     M(\lambda)  &L_{\eta}(\lambda^\ell)^{T}\otimes I_{m}\\\hline
      L_{\epsilon}(\lambda^\ell)\otimes I_{n}& \phantom{\Big{(}} 0 \phantom{\Big{(}}
      \end{array}
    \right]
  \>,
\end{equation}
is called an \emph{$(\epsilon,n,\eta,m)$-block Kronecker degree-$\ell$ matrix polynomial} or, simply, a {\em block Kronecker matrix polynomial} when its size and degree are clear from the context. 
\end{definition}

The following theorem for block Kronecker  matrix polynomials follows immediately as a corollary of the general result in part (b) of Theorem \ref{thm:ell-ification} for strong block minimal bases  matrix polynomials.
\begin{theorem}\label{thm:strong_Kronecker}
Let $\mathcal{L}(\lambda)$ be an $(\epsilon , n, \eta , m)$-block Kronecker degree-$\ell$ matrix polynomial  as in \eqref{eq:Kronecker_ell}.
Then $\mathcal{L}(\lambda)$ is a strong $\ell$-ification of the  matrix polynomial
\begin{equation}
\label{eq:condition}
 (\Lambda_\eta(\lambda^\ell)^T\otimes I_m)M(\lambda)(\Lambda_{\epsilon}(\lambda^\ell)\otimes I_n) \in \mathbb{F}[\lambda]^{m\times n}
\end{equation}
of  grade $\ell(\epsilon+\eta+1)$.
\end{theorem}

Based on block Kronecker matrix polynomials, we now construct companion $\ell$-ifications for $m\times n$ matrix polynomials of grade $d$ for any $\ell<d$, provided that $\ell$ divides $d$, that is, $d=k\ell$, for some non-zero natural number $k$.
Except when $k=2$, these companion $\ell$-ifications are different from the Frobenius-like companion $\ell$-ifications in \cite{IndexSum}.
To this end, let us consider again the matrix polynomials $\{B_i(\lambda)\}_{i=1}^{k}$ defined in \eqref{eq:B_poly} associated with a matrix polynomial $P(\lambda)=\sum_{i=0}^d P_i\lambda^i\in\mathbb{F}[\lambda]^{m\times n}$ of grade $d=k\ell$.
Notice that these polynomials satisfy the equality
\begin{equation}\label{eq:relation_B}
P(\lambda) = \lambda^{k(\ell-1)}B_k(\lambda)+\lambda^{k(\ell-2)}B_{k-1}(\lambda)+\cdots + \lambda^\ell B_2(\lambda)+B_1(\lambda).
\end{equation}
Then, let us fix $\epsilon,\eta\in\mathbb{N}$ such that $\epsilon+\eta=k-1$.
Notice that, except in the case $k=2$, both $\epsilon$ and $\eta$ can be chosen to be nonzero simultaneously (if either $\epsilon$ or $\eta$ is zero the construction that follows produces one of the Frobenius-like $\ell$-ifications).
Let us define the grade-$\ell$ matrix polynomial
\begin{equation}\label{eq:Sigma}
\Sigma_{P}^{(\epsilon,\eta)}(\lambda):=
\begin{bmatrix}
B_k(\lambda) & B_{k-1}(\lambda) & \cdots & B_{\eta+1}(\lambda)\\
0 & \cdots & 0 & \vdots \\
\vdots & \ddots & \vdots & B_2(\lambda) \\
0 & \cdots & 0 & B_1(\lambda)
\end{bmatrix},
\end{equation}
and notice that \eqref{eq:relation_B} implies 
\begin{equation}\label{eq:relation_Sigma}
(\Lambda_\eta(\lambda^\ell)^T\otimes I_m)\Sigma_{P}^{(\epsilon,\eta)}(\lambda)(\Lambda_{\epsilon}(\lambda^\ell)\otimes I_n) = P(\lambda).
\end{equation}
From Theorem \ref{thm:strong_Kronecker}, together with \eqref{eq:relation_Sigma} and the form of \eqref{eq:Sigma}, we readily obtain the following result.
\begin{theorem}\label{thm:companion_forms}
Let $P(\lambda) = \sum_{i=0}^d P_i \lambda^i \in \mathbb{F}[\lambda]^{m \times n}$ be a matrix polynomial with grade $d=k\ell$, for some $k$.
Let $\epsilon + \eta + 1=k$, and let $\Sigma_{P}^{(\epsilon,\eta)}(\lambda)$ be the matrix polynomial in \eqref{eq:Sigma}.
Then, the block Kronecker matrix polynomial
\begin{equation}\label{eq:ell--ification}
  \left[
    \begin{array}{c|c}
    \Sigma_{P}^{(\epsilon,\eta)}(\lambda)  &L_{\eta}(\lambda^\ell)^{T}\otimes I_{m}\\\hline
      L_{\epsilon}(\lambda^\ell)\otimes I_{n}& \phantom{\Big{(}} 0 \phantom{\Big{(}}
      \end{array}
    \right]
\end{equation}
is a companion $\ell$-ification (or a companion form of degree $\ell$) for $m\times n$ matrix polynomials of grade $d$.
\end{theorem}

\begin{example}
Let $P(\lambda) = \sum_{i=0}^6 P_i \lambda^i \in \mathbb{F}[\lambda]^{m \times n}$, then the quadratic matrix polynomial
\[
\begin{bmatrix}
\lambda^2P_6+\lambda P_5&\lambda^2P_4+\lambda P_3&\lambda^2P_2+\lambda P_1+P_0 \\ \hline
-I_n & \lambda^2 I_n & 0 \\ 
0 & -I_n & \lambda^2 I_n
\end{bmatrix}
\]
is the first Frobenius-like companion form introduced in \cite{IndexSum}, and the quadratic matrix polynomial
\[
\left[\begin{array}{cc|c}
\lambda^2P_6+\lambda P_5&\lambda^2P_4+\lambda P_3 & -I_m \\
0 & \lambda^2P_2+\lambda P_1+P_0 & \lambda^2 I_m \\ \hline 
-I_n &\lambda^2 I_n & 0
\end{array}\right]
\]
is the matrix polynomial obtained from Theorem \ref{thm:companion_forms} with $\epsilon=\eta=1$.
Notice that both quadratic matrix polynomials are companion quadratifications for $m\times n$ matrix polynomials of grade $d=6$.
\end{example}

The block Kronecker matrix polynomial \eqref{eq:ell--ification} is clearly not the only block Kronecker matrix polynomial whose (1,1) block $M(\lambda)$ satisfies 
\begin{equation}\label{eq:cond_BK}
(\Lambda_\eta(\lambda^\ell)^T\otimes I_m)M(\lambda)(\Lambda_{\epsilon}(\lambda^\ell)\otimes I_n) = P(\lambda).
\end{equation} 
A succinct characterization of the set of all solutions of \eqref{eq:cond_BK} for a prescribed polynomial $P(\lambda)$ of grade $d$ may be obtained by applying Theorem \ref{thm:solution_final}.
However, for block Kronecker matrix  polynomials, a different but simpler characterization is presented in Theorem \ref{thm:char_Kronecker}.

\begin{theorem}\label{thm:char_Kronecker}
Let $P(\lambda) = \sum_{i=0}^d P_i \lambda^i \in \mathbb{F}[\lambda]^{m \times n}$ be a matrix polynomial with grade $d=k\ell$, for some $k$.
Let $\epsilon + \eta + 1=k$, and let $\Sigma_{P}^{(\epsilon,\eta)}(\lambda)$ be the matrix polynomial in \eqref{eq:Sigma}.
Then, any solution of \eqref{eq:cond_BK} is of the form
\begin{align}\label{eq:charac_BK}
\begin{split}
M(\lambda) = \Sigma_{P}^{(\epsilon,\eta)}(\lambda)+
&\left( \lambda \begin{bmatrix} 0 \\ D(\lambda) \end{bmatrix}+B \right)\left(L_\epsilon(\lambda^\ell)\otimes I_n\right)\\ &+\left(L_\eta(\lambda^\ell)^T \otimes I_m\right)
\left( \lambda \begin{bmatrix} 0 & -D(\lambda) \end{bmatrix}+C\right),
\end{split}
\end{align}
for some matrices $B\in\mathbb{F}^{(\eta+1)m\times \epsilon n}$ and $C\in\mathbb{F}^{\eta m\times (\epsilon+1)n}$, and matrix polynomial $D(\lambda)\in\mathbb{F}_{\ell-2}[\lambda]^{\eta m\times \epsilon n}$.
\end{theorem}
\begin{proof}
Let us introduce the linear operator 
\begin{align*}
\begin{split}
\Xi:\mathbb{F}_{\ell}[\lambda]^{(\eta+1)m\times (\epsilon+1)n}&\longrightarrow \mathbb{F}_{d}[\lambda]^{m\times n}\\
M(\lambda) &\longrightarrow  \Xi[M](\lambda)=(\Lambda_\eta(\lambda^\ell)^T\otimes I_m)M(\lambda)(\Lambda_{\epsilon}(\lambda^\ell)\otimes I_n).
\end{split}
\end{align*}
First, we notice that $M(\lambda)$ as in \eqref{eq:charac_BK} has grade equal to $\ell$.
Since $\Xi[\Sigma_{P}^{(\epsilon,\eta)}](\lambda)=P(\lambda)$, it is easily checked that any $M(\lambda)$ of the form \eqref{eq:charac_BK} satisfies $\Xi[M](\lambda)=P(\lambda)$.
Hence, the linear operator $\Xi$ is surjective and $\mathrm{dim}(\mathrm{null}(\Xi))= \epsilon \eta mn(\ell-1)+(\epsilon+1)\eta mn+(\eta+1)\epsilon mn$, which corresponds to the number of free parameters in \eqref{eq:charac_BK}.
Furthermore, the set of matrix polynomials of the form
\begin{equation}\label{eq:null}
\left( \lambda \begin{bmatrix} 0 \\ D(\lambda) \end{bmatrix}+B \right)\left(L_\epsilon(\lambda^\ell)\otimes I_n\right)+\left(L_\eta(\lambda^\ell)^T \otimes I_m\right)
\left( \lambda \begin{bmatrix} 0 & -D(\lambda) \end{bmatrix}+C\right)
\end{equation}
is contained in $\mathrm{null}(\Xi)$.
Thus, to finish the proof, it suffices to show that the mapping $(B,C,D(\lambda))\rightarrow M(\lambda)$, with $M(\lambda)$ as in \eqref{eq:null}, is injective. 
We show the injectivity of this mapping by showing that the only matrix polynomials $P_1(\lambda)$ and $P_2(\lambda)$ of grade $\ell-1$ satisfying 
\begin{equation}\label{eq:injectivity}
P_1(\lambda)\left(L_\epsilon(\lambda^\ell)\otimes I_n\right)+\left(L_\eta(\lambda^\ell)^T \otimes I_m\right)P_2(\lambda)=0
\end{equation}
are $P_1(\lambda)=0$ and $P_2(\lambda)=0$. 
Indeed, pre-multiplying \eqref{eq:injectivity} by $(\Lambda_\eta(\lambda^\ell)^T\otimes I_m)$ we obtain $(\Lambda_\eta(\lambda^\ell)^T\otimes I_m)P_1(\lambda)(L_\epsilon(\lambda^\ell)\otimes I_n)=0$, which implies $(\Lambda_\eta(\lambda^\ell)^T\otimes I_m)P_1(\lambda)=0$ because $L_\epsilon(\lambda^\ell)\otimes I_n$ has full normal row rank.
Moreover, $(\Lambda_\eta(\lambda^\ell)^T\otimes I_m)P_1(\lambda)=0$ with $P_1(\lambda)\neq 0$ contradicts the fact that all the right minimal indices of $\Lambda_\eta(\lambda^\ell)^T\otimes I_m$ are equal to $\ell$.
Therefore, $P_1(\lambda)=0$.
An analogous argument shows that $P_2(\lambda)=0$.
\end{proof}

Theorem \ref{thm:strong_Kronecker} allows one to easily check whether or not a block Kronecker matrix polynomial is a strong $\ell$-ifications of a prescribed matrix polynomial $P(\lambda)$, and Theorem  \ref{thm:char_Kronecker} allows one to easily construct infinitely many  strong $\ell$-ifications for $P(\lambda)$. 
Moreover, many of these $\ell$-ifications are companion forms different from \eqref{eq:ell--ification} or the Frobenius-like companion forms.
We illustrate this in Example \ref{ex:block_Kron}, where we construct three different block Kronecker matrix polynomials with degrees 1, 2 and 3 that are, respectively, a strong linearization, a strong quadratification, and a strong 3-ification of a given matrix polynomial of grade $d=6$.
Further, the three examples are companion forms.
\begin{example}\label{ex:block_Kron}
Let $P(\lambda)=\sum_{i=0}^6 P_i\lambda^i\in\mathbb{F}[\lambda]^{m\times n}$. 
Then, the following block Kronecker matrix polynomials
\begin{align*}
L(\lambda)=&\left[\begin{array}{cccc|cc}
\lambda P_6 & 0 & 0 & 0 & -I_m & 0 \\
\lambda P_5 & \lambda P_4 & 0 & 0 & \lambda I_m & -I_m \\
0 & \lambda P_3 & \lambda P_2 & \lambda P_1+P_0 & 0 & \lambda I_m \\ \hline
-I_n & \lambda I_n & 0 & 0 & 0 & 0 \\
0 & -I_n & \lambda I_n & 0 & 0 & 0 \\
0 & 0 & -I_n & \lambda I_n & 0 & 0 
\end{array}\right], \\
Q(\lambda)=&\left[\begin{array}{cc|c}
\lambda^2 P_6 + \lambda P_5 & \lambda P_3 + P_2& -I_m \\
\lambda^2 P_4 & \lambda P_1 + P_0 & \lambda^2 I_m \\ \hline
\phantom{\Big{(}}-I_n \phantom{\Big{(}} & \lambda^2 I_n & 0
\end{array}\right], \quad \mbox{and} \\
C(\lambda)=&\left[\begin{array}{cc}
\lambda^3 P_6 + \lambda^2 P_5 + \lambda P_4 + P_3 & \lambda^2 P_2 + \lambda P_1 + P_0 \\ \hline
\phantom{\Big{(}} -I_n \phantom{\Big{(}} & \lambda^3 I_n 
\end{array}\right]
\end{align*} 
are, by Theorem \ref{thm:strong_Kronecker}, respectively, a strong linearization, a strong quadratification, and a strong $3$-ification of the matrix polynomial $P(\lambda)$.
Notice that $L(\lambda)$, $Q(\lambda)$ and $C(\lambda)$ are companion forms for matrix polynomials of grade $6$ and size $m\times n$.
\end{example}

Companion forms may sometimes have other valuable properties in addition to those specified at the beginning of this section.
For example, one may require that the structure of the polynomials is preserved.
We construct in Example \ref{ex:sym} a symmetric companion quadratification for $n\times n$ symmetric matrix polynomials of grade $d=10$.
The construction in \eqref{eq:ex_quad} is easily generalized for any matrix polynomial with grade $d=4k+2$, for some $k$, and to other structures (palindromic, alternating, etc) by using the ideas in \cite{FJP}.
\begin{example}\label{ex:sym}
Let $P(\lambda)=\sum_{i=0}^{10} P_i\lambda^i\in\mathbb{F}[\lambda]^{n\times n}$.
Then, the following block Kronecker quadratic matrix polynomial
\begin{equation}\label{eq:ex_quad}
\left[\begin{array}{ccc|cc}
\lambda^2 P_{10}+\lambda P_9 + P_8 & \lambda P_7/2 & 0 & -I_n & 0 \\
 \lambda P_7/2 & \lambda^2  P_6+\lambda P_5 + \lambda P_4 & \lambda P_3/2 & \lambda^2 I_n & -I_n \\
 0 & \lambda P_3/2 & \lambda^2 P_2+\lambda P_1+P_0 & 0 & \lambda^2 I_n \\  \hline
\phantom{\Big{(}} -I_n \phantom{\Big{(}} & \lambda^2 I_n & 0 & 0 & 0 \\
 0 & -I_n & \lambda^2 I_n & 0 & 0 
 \end{array}\right]
\end{equation}
is, by Theorem \ref{thm:strong_Kronecker}, a strong quadratification of $P(\lambda)$.
Since \eqref{eq:ex_quad} is symmetric when $P(\lambda)$ is symmetric and it is constructed from the coefficients of $P(\lambda)$ without using any arithmetic operation, \eqref{eq:ex_quad} is a symmetric companion quadratification  for $n\times n$ symmetric matrix polynomials of grade $d=10$.
\end{example}

\section{Minimal indices, minimal bases and eigenvector recovery procedures}\label{sec:recovery}
We study in this section how to recover the eigenvectors, and the minimal bases and minimal indices of a matrix polynomial $P(\lambda)$ from those of an $\ell$-ification $\mathcal{L}(\lambda)$ based on strong block minimal bases matrix polynomials. 
When $\mathcal{L}(\lambda)$ is a block Kronecker matrix polynomial, we will see that such eigenvectors and minimal bases recovery procedures are very simple.
More precisely, block Kronecker matrix polynomials allow us to obtain the eigenvectors and the minimal bases of $P(\lambda)$ from those of $\mathcal{L}(\lambda)$ without any extra computational cost.

\subsection{Minimal indices}

It is known that strong $\ell$-ifications may change the minimal indices of a singular matrix polynomial $P(\lambda)$ almost arbitrarily  \cite[Theorem 4.10]{FFP2}. 
For this reason, it is important to be able to recover the minimal indices of $P(\lambda)$ from those of an $\ell$-ification $\mathcal{L}(\lambda)$.
The goal of this section is to show that the minimal indices of the singular matrix polynomial \eqref{eq:Qpolinminbaslin} are related with those of the strong block minimal bases matrix polynomial \eqref{eq:minbas_ell} via uniform shifts.

The following lemma is key to prove Theorem \ref{thm:minimal_indices}, which is the main result of this section.
\begin{lemma} \label{lemm:techindminbaslin}
Let $\mathcal{L}(\lambda)$ be a strong block minimal bases matrix polynomial as in \eqref{eq:minbas_ell}, let $N_1(\lambda)$ be a minimal basis dual to $K_1 (\lambda)$, let $N_2(\lambda)$ be a minimal basis dual to $K_2 (\lambda)$, let $Q(\lambda)$ be the matrix polynomial  in \eqref{eq:Qpolinminbaslin}, and let $\widehat{N}_2 (\lambda)$ be the matrix polynomial appearing in \eqref{eq:twounimodembed}. 
Then the following hold:
\begin{enumerate}
\item[\rm (a)] If $h(\lambda) \in \mathcal{N}_r (Q)$, then
\begin{equation} \label{eq:defz}
z(\lambda) :=
\begin{bmatrix}
N_1 (\lambda)^T \\-\widehat{N}_2 (\lambda) M(\lambda) N_1 (\lambda)^T
\end{bmatrix} h(\lambda) \, \in \mathcal{N}_r (\mathcal{L})\, .
\end{equation}
Moreover, if $0 \ne h(\lambda) \in \mathcal{N}_r (Q)$ is a vector polynomial, then $z(\lambda)$ is also a vector polynomial and
\begin{equation} \label{eq:degreeshift1}
\deg (z (\lambda)) = \deg( N_1 (\lambda)^T \, h(\lambda)) = \deg( N_1 (\lambda)) +  \deg(h(\lambda)).
\end{equation}
\item[\rm (b)] If $\{h_1(\lambda), \ldots, h_p(\lambda)\}$ is a right minimal basis of $Q(\lambda)$, then
\[
\left\{\begin{bmatrix}
N_1 (\lambda)^T \\-\widehat{N}_2 (\lambda) M(\lambda) N_1 (\lambda)^T
\end{bmatrix} h_1(\lambda), \ldots , \begin{bmatrix}
N_1 (\lambda)^T \\-\widehat{N}_2 (\lambda) M(\lambda) N_1 (\lambda)^T
\end{bmatrix} h_p(\lambda) \right\}
\]
is a right minimal basis of $\mathcal{L}(\lambda)$.
\end{enumerate}
\end{lemma}
\begin{proof}
{\bf Proof of part (a)}. Notice that the matrix $X(\lambda)$ in \eqref{eq:XYZminlin} is given by $X(\lambda) = \widehat{N}_2 (\lambda) M(\lambda) N_1 (\lambda)^T$.
 Then, from \eqref{eq:XYZminlin}, we get 
\begin{equation}\label{eq:aux1_techindminbaslin}
(U_2(\lambda)^{-T} \oplus I_{m_1}) \, \mathcal{L}(\lambda) \, (U_1(\lambda)^{-1} \oplus I_{m_2})
\begin{bmatrix}
0 \\ I_n \\ -X(\lambda)
\end{bmatrix}
= \begin{bmatrix}
0 \\ Q(\lambda) \\ 0
\end{bmatrix}.
\end{equation}
Then, by using the structure of $U_1(\lambda)^{-1}$ in \eqref{eq:twounimodembed}, we obtain
\begin{equation}\label{eq:aux2_techindminbaslin}
(U_1(\lambda)^{-1}\oplus I_{m_1})
\begin{bmatrix}
0 \\ I_n \\ -X(\lambda)
\end{bmatrix}=
\begin{bmatrix}
\widehat{N}_1(\lambda)^T & N_1(\lambda)^T & 0 \\
0 & 0 & I_{m_1} 
\end{bmatrix}
\begin{bmatrix}
0 \\ I_n \\ -X(\lambda)
\end{bmatrix}=
\begin{bmatrix}
N_1(\lambda)^T \\ -X(\lambda)
\end{bmatrix}.
\end{equation}
Finally, from \eqref{eq:aux1_techindminbaslin} and \eqref{eq:aux2_techindminbaslin}, we obtain
\begin{equation} \label{eq:ins3forrecovery}
(U_2(\lambda)^{-T} \oplus I_{m_1}) \, \mathcal{L}(\lambda) \,
\begin{bmatrix}
N_1(\lambda)^T \\ -X(\lambda)
\end{bmatrix}
= \begin{bmatrix}
0 \\ Q(\lambda) \\ 0
\end{bmatrix}.
\end{equation}
The above equation implies that $z(\lambda) \in \mathcal{N}_r (\mathcal{L})$ if $h(\lambda) \in \mathcal{N}_r (Q)$.
Also notice that if $h(\lambda)$ is a vector polynomial so is $z(\lambda)$, because $N_1(\lambda)$ and $X(\lambda)$ are matrix polynomials.

To  finish the proof of part (a), it remains to prove the degree shifting property \eqref{eq:degreeshift1}.
To this aim, notice that 
\begin{equation} \label{eq:auxdegreeshift1}
\deg( N_1 (\lambda)^T \, g(\lambda)) = \deg( N_1 (\lambda)) +  \deg(g(\lambda)) \, ,
\end{equation}
for any vector polynomial $g(\lambda) \ne 0$, and that
\begin{equation} \label{eq:auxdegreeshift2}
\deg( K_2 (\lambda)^T \, y(\lambda)) = \deg( K_2 (\lambda)) +  \deg(y(\lambda)) = \ell +  \deg(y(\lambda)) \, ,
\end{equation}
for any vector polynomial $y(\lambda) \ne 0$, since the minimal bases $N_1 (\lambda)$ and $K_2(\lambda)$ both have  constant row degrees and  their highest degree coefficients have full row rank. 
Then, observe that
\begin{equation} \label{eq:maxdegrees}
\deg (z(\lambda)) = \max \{\deg( N_1 (\lambda)^T h(\lambda)) \, , \, \deg (X(\lambda) h(\lambda))\} \, .
\end{equation}
Thus, \eqref{eq:degreeshift1} follows trivially if $X(\lambda) h(\lambda) = 0$.
Finally, assume that $X(\lambda) h(\lambda) \ne 0$ and $h(\lambda) \in \mathcal{N}_r (Q)$. 
Then,  from \eqref{eq:minbas_ell} and \eqref{eq:defz},  and using 
\[
\mathcal{L}(\lambda)z(\lambda) = 
\begin{bmatrix}
M(\lambda) & K_2(\lambda)^T \\ K_1(\lambda) &0
\end{bmatrix}
\begin{bmatrix}
N_1(\lambda)^Th(\lambda) \\ -X(\lambda)h(\lambda)
\end{bmatrix}=0,
\]
we get
\[
M(\lambda) N_1 (\lambda)^T h(\lambda) = K_2(\lambda)^T X(\lambda) h(\lambda).
\]
Taking degrees on both sides of the above equality and using \eqref{eq:auxdegreeshift2}, we obtain
\begin{align*}
\ell + \deg (X(\lambda) h(\lambda)) & = \deg (K_2(\lambda)^T X(\lambda) h(\lambda)) \leq \deg(M(\lambda)) + \deg( N_1 (\lambda)^T h(\lambda)) \\
& \leq \ell + \deg( N_1 (\lambda)^T h(\lambda)),
\end{align*}
and, so, $\deg (X(\lambda) h(\lambda)) \leq \deg( N_1 (\lambda)^T h(\lambda))$.
 This, together with \eqref{eq:auxdegreeshift1} and \eqref{eq:maxdegrees}, proves the degree shifting formula  \eqref{eq:degreeshift1}.

\smallskip

\noindent {\bf Proof of part (b)}. Let us introduce the following matrix polynomial
\[ B(\lambda):=
\begin{bmatrix}
N_1 (\lambda)^T \\-\widehat{N}_2 (\lambda) M(\lambda) N_1 (\lambda)^T
\end{bmatrix} 
\begin{bmatrix} h_1(\lambda) & \cdots & h_p (\lambda)\end{bmatrix}.
\]
First, we prove that the columns of $B(\lambda)$ are a minimal basis of the rational subspace they span by applying \cite[Theorem 2.4]{FFSP}. 
Since $N_1 (\lambda)^T $ and $[h_1(\lambda) \cdots h_p (\lambda)]$ are minimal bases, note that for all $\lambda_0 \in \overline{\mathbb{F}}$, the matrices $N_1 (\lambda_0)^T$ and $[h_1(\lambda_0) \cdots h_p (\lambda_0)]$ have both full column rank (recall \cite[Theorem 2.4]{IndexSum}).
Thus, the matrix $B(\lambda_0)$ has full column rank for all $\lambda_0 \in \overline{\mathbb{F}}$.
 Next, notice that \eqref{eq:degreeshift1} implies that the highest column degree coefficient matrix $B_{hc}$ of $B(\lambda)$ has as a submatrix the highest column degree coefficient matrix $C_{hc}$ of $C(\lambda) := N_1 (\lambda)^T [h_1(\lambda) \cdots h_p (\lambda)]$.
  But since the column degrees of $N_1 (\lambda)^T$ are all equal, we have that $C_{hc}$ is the product of the highest column degree coefficient matrices of $N_1 (\lambda)^T$ and  $[h_1(\lambda) \cdots h_p (\lambda)]$, which have both full column rank because the columns of both matrices are minimal bases. 
  So $C_{hc}$ has full column rank, as well as $B_{hc}$. 
  This implies that the columns of $B(\lambda)$ are a minimal basis of a rational subspace.
  Let us denote this subspace by $\mathcal{S}$. 
Then, by part (a), we get $\mathcal{S} \subseteq \mathcal{N}_r (\mathcal{L})$.
   Finally, since $\mathcal{L}(\lambda)$ is a strong $\ell$-ification of $Q(\lambda)$ by part (b) in Theorem \ref{thm:ell-ification}, we get from \cite[Theorem 4.1]{IndexSum} that $\mathcal{S} = \mathcal{N}_r (\mathcal{L})$ because $\dim(\mathcal{N}_r (Q)) = \dim (\mathcal{N}_r (\mathcal{L}))$.
\end{proof}

As a corollary of Lemma \ref{lemm:techindminbaslin}, we obtain Theorem \ref{thm:minimal_indices}, which shows that the minimal indices of the strong block minimal bases  matrix polynomial \eqref{eq:minbas_ell} are related with those of the polynomial $Q(\lambda)$  in \eqref{eq:Qpolinminbaslin} via uniform shifts.
\begin{theorem}\label{thm:minimal_indices}
Let $\mathcal{L}(\lambda)$ be a strong block minimal bases degree-$\ell$ matrix polynomial as in  \eqref{eq:minbas_ell}, let $N_1(\lambda)$ be a minimal basis dual to $K_1 (\lambda)$, let $N_2(\lambda)$ be a minimal basis dual to $K_2 (\lambda)$, and let $Q(\lambda)$ be the matrix polynomial in \eqref{eq:Qpolinminbaslin}. 
Then, the following statements hold:
\begin{enumerate}
\item[\rm (a)] If $0 \leq \epsilon_1 \leq \epsilon_2 \leq \cdots \leq \epsilon_p$ are the right minimal indices of $Q(\lambda)$, then
\[
\epsilon_1 + \deg(N_1(\lambda)) \leq \epsilon_2 + \deg(N_1(\lambda)) \leq \cdots \leq \epsilon_p +  \deg(N_1(\lambda))
\]
are the right minimal indices of $\mathcal{L}(\lambda)$.
\item[\rm (b)] If $0 \leq \eta_1 \leq \eta_2 \leq \cdots \leq \eta_q$ are the left minimal indices of $Q(\lambda)$, then
\[
\eta_1 + \deg(N_2(\lambda)) \leq \eta_2 + \deg(N_2(\lambda)) \leq \cdots \leq \eta_q +  \deg(N_2(\lambda))
\]
are the left minimal indices of $\mathcal{L}(\lambda)$.
\end{enumerate}
\end{theorem}
\begin{proof} Part (a) follows from combining part (b) in Lemma \ref{lemm:techindminbaslin} with equation \eqref{eq:degreeshift1}. 
To prove part (b), we recall that the left minimal indices of $Q(\lambda)$ are the right minimal indices of $Q(\lambda)^T$.
Additionally, from Lemma \ref{lemma:transpose} and its proof, we have that $\mathcal{L}(\lambda)^T$ is also a strong block minimal bases degree-$\ell$ matrix polynomial (with the roles of $(K_1(\lambda), N_1(\lambda))$ and $(K_2(\lambda),N_2 (\lambda))$ interchanged) that is a strong $\ell$-ification of $Q(\lambda)^T$.
Hence, part (b) follows from applying part (a) to the right minimal indices of $\mathcal{L}(\lambda)^T$ and $Q(\lambda)^T$.
\end{proof}

Clearly, one can apply Theorem \ref{thm:minimal_indices} to a block Kronecker matrix polynomial as in \eqref{eq:Kronecker_ell}.
\begin{theorem}
Let $P(\lambda) = \sum_{i=0}^d P_i \lambda^i \in \mathbb{F}[\lambda]^{n\times n}$ be a singular matrix polynomial with $d=k\ell$, for some $k$.
Let $\mathcal{L}(\lambda)$ be an $(\epsilon,n,\eta,n)$-block Kronecker degree-$\ell$ matrix polynomial as in \eqref{eq:Kronecker_ell} with $k = \epsilon + \eta + 1$ such that
$P(\lambda) = (\Lambda_\eta(\lambda^\ell)^T\otimes I_m)M(\lambda)(\Lambda_{\epsilon}(\lambda^\ell)\otimes I_n)$, where $\Lambda_k (\lambda)$ is the vector polynomial in \eqref{eq:Lambda}. 
Then the following hold:
\begin{enumerate}
\item[\rm (a)] If $0 \leq \epsilon_1 \leq \epsilon_2 \leq \cdots \leq \epsilon_p$ are the right minimal indices of $P(\lambda)$, then
\[
\epsilon_1 + \epsilon\ell \leq \epsilon_2 + \epsilon\ell \leq \cdots \leq \epsilon_p +  \epsilon\ell
\]
are the right minimal indices of $\mathcal{L}(\lambda)$.
\item[\rm (b)] If $0 \leq \eta_1 \leq \eta_2 \leq \cdots \leq \eta_q$ are the left minimal indices of $P(\lambda)$, then
\[
\eta_1 + \eta\ell \leq \eta_2 + \eta\ell \leq \cdots \leq \eta_q +  \eta\ell
\]
are the left minimal indices of $\mathcal{L}(\lambda)$.
\end{enumerate}
\end{theorem}

\subsection{Minimal bases recovery procedures}

We discuss in this section how to recover the minimal bases of the singular matrix polynomial \eqref{eq:Qpolinminbaslin} from those of the singular strong block minimal bases matrix polynomial  \eqref{eq:minbas_ell}.
In particular, we show that such recovery procedures allow us to obtain  the minimal bases of the polynomial  \eqref{eq:Qpolinminbaslin} from those of any block Kronecker matrix polynomial without any extra computational cost.


The first result is Lemma \ref{lemm:SECONDtechindminbaslin}, valid for any strong block minimal bases matrix polynomial, which completes Lemma \ref{lemm:techindminbaslin} and uses the notation introduced in Section \ref{sec:min_ell-ifications}.
Lemma \ref{lemm:SECONDtechindminbaslin} gives abstract formulas for the minimal bases of block minimal bases matrix polynomials \eqref{eq:minbas_ell} in terms of those of the matrix polynomial \eqref{eq:Qpolinminbaslin}.
\begin{lemma} \label{lemm:SECONDtechindminbaslin}
Let $\mathcal{L}(\lambda)$ be a strong block minimal bases degree-$\ell$ matrix polynomial as in \eqref{eq:minbas_ell}, let $N_1(\lambda)$ be a minimal basis dual to $K_1 (\lambda)$, let $N_2(\lambda)$ be a minimal basis dual to $K_2 (\lambda)$, let $Q(\lambda)$ be the matrix polynomial in \eqref{eq:Qpolinminbaslin}, and let $\widehat{N}_1 (\lambda)$ and $\widehat{N}_2 (\lambda)$ be the matrices appearing in \eqref{eq:twounimodembed}. 
Then, the following statements hold.
\begin{enumerate}
\item[\rm (a)] Any right minimal basis of $\mathcal{L}(\lambda)$ has the form
\[
\left\{\begin{bmatrix}
N_1 (\lambda)^T \\-\widehat{N}_2 (\lambda) M(\lambda) N_1 (\lambda)^T
\end{bmatrix} h_1(\lambda), \ldots , \begin{bmatrix}
N_1 (\lambda)^T \\-\widehat{N}_2 (\lambda) M(\lambda) N_1 (\lambda)^T
\end{bmatrix} h_p(\lambda) \right\},
\]
where $\{h_1(\lambda), \ldots, h_p(\lambda)\}$ is some right minimal basis of $Q(\lambda)$.

\item[\rm (b)] Any left minimal basis of $\mathcal{L}(\lambda)$ has the form
\end{enumerate}
\[
\left\{g_1(\lambda)^T [
N_2 (\lambda) , -  N_2 (\lambda) M(\lambda)  \widehat{N}_1 (\lambda)^T
], \ldots ,  g_q(\lambda)^T [
N_2 (\lambda) , -  N_2 (\lambda) M(\lambda)  \widehat{N}_1 (\lambda)^T
] \right\},
\]
\phantom{aaaaaai} where $\{g_1(\lambda)^T, \ldots, g_q(\lambda)^T\}$ is some left minimal basis of $Q(\lambda)$.
\end{lemma}
\begin{proof} We only prove part (a), since part (b) follows from applying part (a) to $\mathcal{L}(\lambda)^T$ and $Q(\lambda)^T$ and then taking transposes, as in the proof of part (b) in Theorem \ref{thm:minimal_indices}.

According to part (b) in Lemma \ref{lemm:techindminbaslin}, if the $p$ columns of $R(\lambda)$ are a particular right minimal basis of $Q(\lambda)$, then the columns of
\begin{equation} \label{eq:Ssec7}
S(\lambda) := \begin{bmatrix}
N_1 (\lambda)^T \\-\widehat{N}_2 (\lambda) M(\lambda) N_1 (\lambda)^T
\end{bmatrix} R(\lambda)
\end{equation}
are a particular right minimal basis of $\mathcal{L}(\lambda)$. Therefore, any other right minimal basis $\mathcal{B}_\mathcal{L}$ of $\mathcal{L}(\lambda)$ has the form
\begin{equation} \label{eq:Bsec7}
\mathcal{B}_{\mathcal{L}} = \{S(\lambda) v_1 (\lambda) , \ldots , S(\lambda) v_p (\lambda) \},
\end{equation}
where $v_1 (\lambda), \ldots, v_p(\lambda)$ are vector polynomials, and where it is assumed without loss of generality that $\deg(S(\lambda) v_1 (\lambda)) \leq \allowbreak \cdots \allowbreak \leq \deg(S(\lambda) v_p (\lambda))$. 
The minimality of $\mathcal{B}_{\mathcal{L}}$ follows from the fact that  the columns of $S(\lambda)$ are a minimal basis (see \cite[Part 4 in Main Theorem, p. 495]{Forney}).
Hence, it suffices to prove that $\mathcal{B}_{Q} := \{R(\lambda) v_1 (\lambda) , \ldots , R(\lambda) v_p (\lambda) \}$ is a minimal basis of $\mathcal{N}_r(Q)$.
Indeed, first we note that the vector polynomial $R(\lambda) v_i(\lambda) \in \mathcal{N}_r(Q)$, for $i=1,\ldots , p$, from \eqref{eq:ins3forrecovery}.
Second, the set  $\{R(\lambda) v_1 (\lambda) , \ldots , R(\lambda) v_p (\lambda) \}$ is linearly independent because $\mathcal{B}_{\mathcal{L}}$ is linearly independent.
Third, the set $\mathcal{B}_{Q}$ is a polynomial basis of $\mathcal{N}_r (Q)$ since $\dim \mathcal{N}_r (Q) = \dim \mathcal{N}_r (\mathcal{L})$.
Finally, part (a) in Lemma \ref{lemm:techindminbaslin} implies that $\deg(S(\lambda) v_i (\lambda)) =  \deg(N_1(\lambda)) + \deg (R(\lambda) v_i (\lambda))$, and part (a) in Theorem \ref{thm:minimal_indices} implies that $\deg(S(\lambda) v_i (\lambda))  = \deg(N_1(\lambda)) + \epsilon_i$, for $i=1,\ldots , p$, where $\epsilon_1 \leq \cdots \leq \epsilon_p$ are the right minimal indices of $Q(\lambda)$. 
 Therefore, $\deg (R(\lambda) v_i (\lambda)) = \epsilon_i$, which means that $\mathcal{B}_{Q}$ is a right minimal basis of $Q(\lambda)$.
\end{proof}

As a corollary of Lemma \ref{lemm:SECONDtechindminbaslin}, we obtain Theorem \ref{thm:recovery_min_bases}, which shows abstract recovery results for minimal bases  of matrix polynomials from strong block minimal bases matrix polynomials.
\begin{theorem}\label{thm:recovery_min_bases}
Let $\mathcal{L}(\lambda)$ be a strong block minimal bases degree-$\ell$ matrix polynomial as in \eqref{eq:minbas_ell}, let $\widehat{K}_1 (\lambda)$ and $\widehat{K}_2 (\lambda)$ be the matrices appearing in \eqref{eq:twounimodembed}, let the $p$ columns of $R_\mathcal{L}(\lambda)$  be  a right minimal basis of $\mathcal{L}(\lambda)$, and let the $p$ rows of $L_\mathcal{L}(\lambda)$ be a left minimal basis of $\mathcal{L}(\lambda)$. 
Then, the following statements hold.
\begin{itemize}
\item[\rm (a)] The $p$ columns of 
\[
R_Q(\lambda) :=
\begin{bmatrix}
\widehat{K}_1(\lambda) & 0
\end{bmatrix}
R_\mathcal{L}(\lambda)
\]
are a right minimal basis of $Q(\lambda)$.
\item[\rm (b)] The $p$ rows of 
\[
L_Q(\lambda):=L_\mathcal{L}(\lambda)
\begin{bmatrix}
\widehat{K}_2(\lambda)^T \\ 0
\end{bmatrix}
\]
are a left minimal basis of $Q(\lambda)$.
\end{itemize}
\end{theorem}
\begin{proof}
As in the proof of Lemma \ref{lemm:SECONDtechindminbaslin}, we only prove part (a), since part (b) follows from applying part (a) to $\mathcal{L}(\lambda)^T$ and $Q(\lambda)^T$, together with Lemma \ref{lemma:transpose}.

According to part (a) in Lemma \ref{lemm:SECONDtechindminbaslin}, any right minimal basis of $\mathcal{L}(\lambda)$ is of the form
\[
R_\mathcal{L}(\lambda) = \begin{bmatrix}
N_1(\lambda)^T \\
-\widehat{N}_2(\lambda)M(\lambda)N_1(\lambda)^T
\end{bmatrix}R_Q(\lambda),
\]
for some right minimal basis $R_Q(\lambda)$ of $Q(\lambda)$.
Then, the result follows immediately from $\widehat{K}_1(\lambda)N_1(\lambda)^TR_Q(\lambda)=R_Q(\lambda)$.
\end{proof}

In general, the minimal bases recovery procedure in Theorem \ref{thm:recovery_min_bases} requires one matrix-vector multiplication for each vector of the basis. 
Thus, the potential simplicity and low computational cost of these recovery results and formulas depend on the particular strong block minimal bases matrix polynomials used.
In the particular case of block Kronecker matrix polynomials, the minimal bases recovery procedures turn to be particularly simple, as we show in the following theorem.
\begin{theorem}\label{thm:recovery_min_bases_Kron}
Let $P(\lambda) = \sum_{i=0}^d P_i \lambda^i \in \mathbb{F}[\lambda]^{m\times n}$ be a singular matrix polynomial with $d=k\ell$, for some $k$. 
Let $\mathcal{L}(\lambda)$ be an $(\epsilon,n,\eta,m)$-block Kronecker degree-$\ell$ matrix polynomial as in \eqref{eq:Kronecker_ell} with $k = \epsilon + \eta + 1$ such that
$P(\lambda) = (\Lambda_\eta(\lambda^\ell)^T\otimes I_m)M(\lambda)(\Lambda_{\epsilon}(\lambda^\ell)\otimes I_n)$, where $\Lambda_k (\lambda)$ is the vector polynomial in \eqref{eq:Lambda}. 
Consider the pencil $\mathcal{L}(\lambda)$ partitioned into $k\times k$ blocks whose sizes are fixed by the size $n\times n$ of the blocks of $L_\epsilon (\lambda^\ell) \otimes I_n$ and the size $m\times m$ of the blocks of $L_\eta (\lambda^\ell)^T \otimes I_m$. 
Then, the following statements hold.
\begin{enumerate}
\item[{\rm (a)}] If $\{z_1(\lambda),z_2(\lambda), \hdots, z_p(\lambda)\}$ is any right minimal basis of $\mathcal{L}(\lambda)$ whose vectors are partitioned into blocks conformable to the block columns of $\mathcal{L}(\lambda)$, and if $x_j(\lambda)$ is the $(\epsilon+1)th$ block of $z_j(\lambda)$, for $j=1,2,\hdots,p$, then $\{x_1(\lambda),x_2(\lambda),\hdots, \allowbreak x_p(\lambda)\}$ is a right minimal basis of $P(\lambda)$.
\item[{\rm (b)}] If $\{w_1(\lambda)^T,w_2(\lambda)^T, \hdots, w_q(\lambda)^T\}$ is any left minimal basis of $\mathcal{L}(\lambda)$ whose vectors are partitioned into blocks conformable to the block rows of $\mathcal{L}(\lambda)$, and if $y_j(\lambda)^T$ is the $(\eta+1)th$ block of $w_j(\lambda)^T$, for $j=1,2,\hdots,q$,
then $\{y_1(\lambda)^T,\allowbreak y_2(\lambda)^T,\hdots,y_q(\lambda)^T\}$ is a left minimal basis of $P(\lambda)$.
\end{enumerate}
\end{theorem}
\begin{proof}
Parts (a) and (b) are an immediate consequence of Lemma \ref{lemm:SECONDtechindminbaslin} combined with the fact that for $(\epsilon,n,\eta,m)$-block Kronecker matrix polynomials $N_1(\lambda)=\Lambda_\epsilon(\lambda^\ell)^T\otimes I_n$ and $N_2(\lambda)=\Lambda_\eta(\lambda^\ell)^T\otimes I_m$.
\end{proof}
\subsection{Eigenvectors recovery procedures}

In this subsection, we present the final recovery procedures for eigenvectors of regular matrix polynomials from those of block minimal bases matrix polynomials and, in particular, from those of block Kronecker matrix polynomials.

Lemma \ref{lemm:THIRDtechindminbaslin} is the counterpart of part (b) in Lemma \ref{lemm:techindminbaslin} and Lemma \ref{lemm:SECONDtechindminbaslin} for eigenvectors. 
Since only regular matrix polynomials have eigenvectors, we assume that $Q(\lambda)$ is square and regular, which is equivalent to the fact that $\mathcal{L}(\lambda)$ is square and regular, since $\mathcal{L}(\lambda)$ is a strong $\ell$-ification of $Q(\lambda)$.
 The proof of Lemma \ref{lemm:THIRDtechindminbaslin} is omitted for brevity, and because it follows the same steps as those of Lemmas \ref{lemm:techindminbaslin} and \ref{lemm:SECONDtechindminbaslin} but removing all the arguments concerning degrees since only null spaces of constant matrices are considered. 
 Nevertheless, we emphasize that the key tool for proving Lemma \ref{lemm:THIRDtechindminbaslin} is equation \eqref{eq:ins3forrecovery} evaluated at the eigenvalue $\lambda_0$ of interest.

\begin{lemma} \label{lemm:THIRDtechindminbaslin}
Let $\mathcal{L}(\lambda)$ be a square and regular strong block minimal bases degree-$\ell$ matrix polynomial as in \eqref{eq:minbas_ell}, let $N_1(\lambda)$ be a minimal basis dual to $K_1 (\lambda)$, let $N_2(\lambda)$ be a minimal basis dual to $K_2 (\lambda)$, let $Q(\lambda)$ be the matrix polynomial in \eqref{eq:Qpolinminbaslin}, and let $\widehat{N}_1 (\lambda)$ and $\widehat{N}_2 (\lambda)$ be the matrices appearing in \eqref{eq:twounimodembed}. 
Let $\lambda_0$ be a finite eigenvalue of $Q(\lambda)$ (which is also an eigenvalue of $\mathcal{L}(\lambda)$).
 Then, the following statements hold.
\begin{enumerate}
\item[\rm (a)] Let $G_1(\lambda_0):=N_1 (\lambda_0)  M(\lambda_0)^T\widehat{N}_2 (\lambda_0)^T$. 
Then, any basis of $\mathcal{N}_r (\mathcal{L}(\lambda_0))$ has the form
\[ \mathcal{B}_{r,\lambda_0} =
\left\{\begin{bmatrix}
N_1 (\lambda_0)^T \\-G_1(\lambda)^T
\end{bmatrix} x_1, \ldots , \begin{bmatrix}
N_1 (\lambda_0)^T \\-G_1(\lambda_0)^T
\end{bmatrix} x_t \right\},
\]
where $\{x_1 , \ldots, x_t\}$ is some basis of $\mathcal{N}_r(Q(\lambda_0))$, and, vice versa, for any basis $\{x_1 , \ldots, x_t\}$ of $\mathcal{N}_r(Q(\lambda_0))$,  the set of vectors $\mathcal{B}_{r,\lambda_0}$ is a basis of $\mathcal{N}_r (\mathcal{L}(\lambda_0))$.

\item[\rm (b)] Let $G_2(\lambda_0) := N_2 (\lambda_0) M(\lambda_0)  \widehat{N}_1 (\lambda_0)^T$. 
Then, any basis of $\mathcal{N}_\ell (\mathcal{L}(\lambda_0))$ has the form
\[ \mathcal{B}_{\ell,\lambda_0} =
\left\{y_1^T [
N_2 (\lambda_0) , -  G_2(\lambda_0)
], \ldots ,  y_t^T [
N_2 (\lambda_0) , - G_2(\lambda_0)
] \right\},
\]
where $\{y_1^T, \ldots, y_t^T\}$ is some basis of $\mathcal{N}_\ell (Q(\lambda_0))$, and, vice versa, for any basis $\{y_1^T , \ldots, y_t^T\}$ of $\mathcal{N}_\ell (Q(\lambda_0))$,  the set of vectors $\mathcal{B}_{\ell,\lambda_0}$ is a basis of $\mathcal{N}_\ell (\mathcal{L}(\lambda_0))$.
\end{enumerate}
\end{lemma}

As a corollary of Lemma \ref{lemm:THIRDtechindminbaslin}, we obtain Theorem \ref{thm:recovery_eigenvectors}, which shows abstract recovery results for eigenvectors  of matrix polynomials from those of strong block minimal bases matrix polynomials.
The proof of Theorem \ref{thm:recovery_eigenvectors} follows the same steps as those of Theorem \ref{thm:recovery_min_bases}, so it is omitted.
 For brevity, only the recovery of individual eigenvectors is explicitly stated. 
Clearly, the recovery of bases of the corresponding null spaces follows the same pattern.
\begin{theorem}\label{thm:recovery_eigenvectors}
Let $\mathcal{L}(\lambda)$ be a strong block minimal bases degree-$\ell$ matrix polynomial as in \eqref{eq:minbas_ell}, let $\widehat{K}_1 (\lambda)$ and $\widehat{K}_2 (\lambda)$ be the matrices appearing in \eqref{eq:twounimodembed}, let $\lambda_0$ be a finite eigenvalue of $\mathcal{L}(\lambda)$, and let $z$ and $w^T$ be, respectively,  right and  left eigenvectors of $\mathcal{L}(\lambda)$ associated with $\lambda_0$.
Then, the following statements hold.
\begin{itemize}
\item[\rm (a)] The vector
\[
x :=
\begin{bmatrix}
\widehat{K}_1(\lambda_0) & 0
\end{bmatrix}
z
\]
is a right eigenvector of $Q(\lambda)$ for the eigenvalue $\lambda_0$.
\item[\rm (b)] The vector
\[
y^T:=w^T
\begin{bmatrix}
\widehat{K}_2(\lambda_0)^T \\ 0
\end{bmatrix}
\]
is a left eigenvector of $Q(\lambda)$ for the eigenvalue $\lambda_0$.
\end{itemize}
\end{theorem}

As with the minimal bases recovery procedure in Theorem \ref{thm:recovery_min_bases}, the eigenvectors recovery procedure in Theorem \ref{thm:recovery_eigenvectors} requires one matrix-vector multiplication for each eigenvector.
Thus, its potential simplicity and low computational cost depend on the particular strong block minimal bases matrix polynomials used.
In the particular case of   block Kronecker matrix polynomials, the eigenvectors recovery procedures turn to be particularly simple, as we show in the following theorem.
\begin{theorem}\label{thm:finite_eig_Kron}
Let $P(\lambda) = \sum_{i=0}^d P_i \lambda^i \in \mathbb{F}[\lambda]^{n\times n}$ be a regular matrix polynomial with $d=k\ell$, for some $k$.
Let $\mathcal{L}(\lambda)$ be an $(\epsilon,n,\eta,n)$-block Kronecker matrix polynomial as in \eqref{eq:Kronecker_ell} with $k = \epsilon + \eta + 1$ such that
$P(\lambda) = (\Lambda_\eta(\lambda^\ell)^T\otimes I_n)M(\lambda)(\Lambda_{\epsilon}(\lambda^\ell)\otimes I_n)$, where $\Lambda_k (\lambda)$ is the vector polynomial in \eqref{eq:Lambda}. 
Consider the pencil $\mathcal{L}(\lambda)$ partitioned into $k\times k$ blocks of size $n\times n$, any vector of size $nk\times 1$  partitioned into $k\times 1$ blocks of size $n\times 1$, and any vector of size $1 \times nk$  partitioned into $1\times k$ blocks of size $1\times n$. 
Then the following statements hold.
\begin{enumerate}
\item[{\rm (a1)}] If $z\in\mathbb{F}^{nk\times 1}$ is a right eigenvector of $\mathcal{L}(\lambda)$ with finite eigenvalue $\lambda_0$, then the $(\epsilon+1)$th block of $z$ is a right eigenvector of $P(\lambda)$ with finite eigenvalue $\lambda_0$.
\item[{\rm (a2)}] If $z\in\mathbb{F}^{nk\times 1}$ is a right eigenvector of $\mathcal{L}(\lambda)$ for the eigenvalue $\infty$, then the first block of $z$ is a right eigenvector of $P(\lambda)$ for the eigenvalue $\infty$.
\item[{\rm (b1)}]  If $w^T\in\mathbb{F}^{1\times nk}$ is a left eigenvector of $\mathcal{L}(\lambda)$ with finite eigenvalue $\lambda_0$, then the $(\eta+1)$th block of $w^T$ is a left eigenvector of $P(\lambda)$ with finite eigenvalue $\lambda_0$.
\item[{\rm (b2)}]  If $w^T\in\mathbb{F}^{1\times nk}$ is a left eigenvector of $\mathcal{L}(\lambda)$ for the eigenvalue $\infty$,  then the first block of $w^T$ is a left eigenvector of $P(\lambda)$ for the eigenvalue $\infty$.
\end{enumerate}
\end{theorem}

\begin{proof} Parts (a1) and (b1) follow directly from Lemma \ref{lemm:THIRDtechindminbaslin} and Theorem \ref{thm:recovery_eigenvectors},
just by taking into account that for an $(\epsilon,n,\eta,n)$-block Kronecker matrix polynomial $N_1 (\lambda)^T = \Lambda_\epsilon (\lambda^\ell) \otimes I_n$ and $N_2 (\lambda) = \Lambda_\eta (\lambda^\ell)^T \otimes I_n$.

In order to prove parts (a2) and (b2), recall that the eigenvectors of $\mathcal{L}(\lambda)$ (resp. $P(\lambda)$) corresponding to the eigenvalue $\infty$ are those of $\rev_\ell \mathcal{L}(\lambda)$ (resp. $\rev_d P(\lambda)$) corresponding to the eigenvalue $0$. 
As a consequence of Theorem \ref{thm:minbasesdegreeseq}, the matrix polynomial $\rev_\ell \mathcal{L}(\lambda)$ is a strong block minimal bases matrix polynomial (although not exactly a block Kronecker matrix polynomial), which is an $\ell$-ification of $\rev_dP(\lambda)$ (recall the proof of part (b) of Theorem \ref{thm:ell-ification}).
 Therefore, Lemma \ref{lemm:THIRDtechindminbaslin} and Theorem \ref{thm:recovery_eigenvectors} can be applied to the zero eigenvalue of $\rev_\ell \mathcal{L}(\lambda)$ and $\rev_d P(\lambda)$.
 For doing this properly, $N_1(\lambda_0)^T$ has to be replaced by $\rev_{\epsilon\ell} \Lambda_\epsilon (\lambda_0^\ell)\otimes I_n$ and $N_2(\lambda_0)$ has to be replaced by $\rev_{\eta\ell} \Lambda_\eta (\lambda_0^\ell)^T\otimes I_m$, as a consequence of  Theorem \ref{thm:minbasesdegreeseq} (together with other replacements which are of no interest in this proof). 
Then, parts (a2) and (b2) follow from the expression of $\rev_{\epsilon\ell} \Lambda_\epsilon (\lambda^\ell)\otimes I_n$ and  $\rev_{\eta\ell} \Lambda_\eta (\lambda^\ell)^T\otimes I_m$.
\end{proof}

\subsection{One-sided factorizations}

We end this section by showing that strong block minimal bases matrix polynomials admit one-sided factorizations as those used in \cite{Framework}.
One-sided factorizations are useful for performing residual ``local'', i.e., for each particular computed eigenpair, backward error analyses of regular PEPs solved by $\ell$-ifications.

In the following definition we introduce right- and left-sided factorizations as were defined in \cite{Framework} particularized to the case of matrix polynomials.
\begin{definition}\label{def:one-sided}
Given two matrix polynomials $M(\lambda)\in\mathbb{F}[\lambda]^{r\times r}$ and $N(\lambda)\in\mathbb{F}[\lambda]^{s\times s}$, we say that $M(\lambda)$ and $N(\lambda)$ satisfy a right-sided factorization if
\[
M(\lambda)F(\lambda)=G(\lambda)N(\lambda)
\]
for some matrix polynomials $F(\lambda),G(\lambda)\in\mathbb{F}[\lambda]^{r\times s}$.
Additionally, we say that $M(\lambda)$ and $N(\lambda)$ satisfy a left-sided factorization if
\[
E(\lambda)M(\lambda)=N(\lambda)H(\lambda),
\]
for some matrix polynomials $E(\lambda),H(\lambda)\in\mathbb{F}[\lambda]^{s\times r}$.
\end{definition}

\begin{theorem}\label{thm:one-sided}
Let $\mathcal{L}(\lambda)$ be a strong block minimal bases matrix polynomial as in \eqref{eq:minbas_ell}, let $N_1(\lambda)$ be a minimal basis dual to $K_1 (\lambda)$, let $N_2(\lambda)$ be a minimal basis dual to $K_2 (\lambda)$, let $Q(\lambda)$ be the matrix polynomial in \eqref{eq:Qpolinminbaslin}, let $\widehat{K}_i (\lambda)$ and $\widehat{N}_i (\lambda)$ be the matrix polynomials appearing in \eqref{eq:twounimodembed}, for $i=1,2$.
Then, the following right- and left-sided factorizations hold.
\[
\mathcal{L}(\lambda)
\begin{bmatrix}
N_1(\lambda)^T \\
-\widehat{N}_2(\lambda)M(\lambda)N_1(\lambda)^T
\end{bmatrix}=
\begin{bmatrix}
\widehat{K}_2(\lambda)^T \\ 0
\end{bmatrix}Q(\lambda),
\]
and
\[
\begin{bmatrix}
 N_2(\lambda) & -N_2(\lambda)M(\lambda)\widehat{N}_1(\lambda)^T
\end{bmatrix}\mathcal{L}(\lambda)=
Q(\lambda)\begin{bmatrix}
\widehat{K}_1(\lambda) & 0
\end{bmatrix}.
\]
\end{theorem}
\begin{proof}
From \eqref{eq:ins3forrecovery}, we obtain
\[
\mathcal{L}(\lambda)
\begin{bmatrix}
N_1(\lambda)^T \\
-\widehat{N}_2(\lambda)M(\lambda)N_1(\lambda)^T
\end{bmatrix}=(U_2(\lambda)^T\oplus I_{m_1})
\begin{bmatrix}
0 \\ Q(\lambda) \\ 0
\end{bmatrix}
\]
which, by using the structure of $U_2(\lambda)$ in \eqref{eq:twounimodembed}, implies the right-sided factorization.
The left-sided factorization is obtained from a right-sided factorization of $\mathcal{L}(\lambda)^T$.
\end{proof}

In many important situations (see, for example, \cite{conditioning,BackErrors}), the one-sided factorizations in Definition \ref{def:one-sided} typically hold in the more specialized forms
\[
M(\lambda)F(\lambda)=g\otimes N(\lambda) \quad \mbox{and} \quad E(\lambda)M(\lambda)=h^T\otimes N(\lambda),
\]
for some vectors $g,h$.
In the following theorem, we show that for block Kronecker matrix polynomials  the one-sided factorization in Theorem \ref{thm:one-sided}  take this simpler form.
The matrix polynomial (recall Example \ref{ex:unimodular})
\[
\widehat{\Lambda}_k(\lambda) := \left[ \begin{array}{ccccc}
-1 & -\lambda & -\lambda^{2}& \cdots & -\lambda^{k-1} \\
& -1 & -\lambda & \ddots & \vdots \\
& & -1 & \ddots & -\lambda^{2}\\
& & & \ddots & -\lambda \\
& & & & -1 \\
0 & \cdots & \cdots & \cdots & 0
\end{array}\right]
\]
is important in what follows.
\begin{theorem}
Let $P(\lambda) = \sum_{i=0}^d P_i \lambda^i \in \mathbb{F}[\lambda]^{n\times n}$ be a regular matrix polynomial with $d=k\ell$, for some $k$.
Let $\mathcal{L}(\lambda)$ be an $(\epsilon,n,\eta,n)$-block Kronecker matrix polynomial as in \eqref{eq:Kronecker_ell} with $k = \epsilon + \eta + 1$ such that
$P(\lambda) = (\Lambda_\eta(\lambda^\ell)^T\otimes I_n)M(\lambda)(\Lambda_{\epsilon}(\lambda^\ell)\otimes I_n)$, where $\Lambda_k (\lambda)$ is the vector polynomial in \eqref{eq:Lambda}. 
Then, the following right- and left-sided factorizations hold:
\[
\mathcal{L}(\lambda)
\begin{bmatrix}
\Lambda_\epsilon(\lambda^\ell)\otimes I_n \\
-(\widehat{\Lambda}_\eta(\lambda^\ell)^T\otimes I_m)(\lambda)M(\lambda)(\Lambda_\epsilon(\lambda^\ell)\otimes I_n)
\end{bmatrix}=
\begin{bmatrix}
e_{\eta+1} \\ 0
\end{bmatrix}\otimes P(\lambda),
\]
and
\[
\begin{bmatrix}
\Lambda_\eta(\lambda^\ell)^T\otimes I_m & -(\Lambda_\eta(\lambda^\ell)^T\otimes I_m)M(\lambda)(\widehat{\Lambda}_\epsilon(\lambda^\ell)\otimes I_n)
\end{bmatrix}\mathcal{L}(\lambda)=
\begin{bmatrix}
e_{\epsilon+1} \\ 0
\end{bmatrix}^T\otimes P(\lambda),
\]
where $e_{\eta+1}$ and $e_{\epsilon+1}$ denote, respectively, the last columns of the identity matrices $I_{\eta+1}$ and $I_{\epsilon+1}$.
\end{theorem}
\begin{proof}
The one-sided factorizations follow immediately from Theorem \ref{thm:one-sided} taking into account that $\widehat{K}_1(\lambda)=e_{\epsilon+1}\otimes I_n$, $\widehat{K}_2(\lambda)=e_{\eta+1}\otimes I_m$, $\widehat{N}_1(\lambda)=\widehat{\Lambda}_\epsilon(\lambda^\ell)^T\otimes I_n$ and $\widehat{N}_2(\lambda)=\widehat{\Lambda}_\eta(\lambda^\ell)^T\otimes I_m$.
\end{proof}

\section{Conclusions}

We have introduced the family of strong block minimal bases matrix polynomials.
This family has allowed us to present a new procedure for constructing strong $\ell$-ifications for $m\times n$ matrix polynomials of grade $g$, which unifies and extends previous constructions \cite{FFP,FPJP}. 
This procedure is valid in the case where $\ell$ divides $nd$ or $md$.
Any strong $\ell$-ification obtained from this approach presents many properties that are desirable for numerical computations.
First, the $\ell$-ification is constructed using
only simple operations on the coefficients of the matrix polynomial.
Second, the left and right minimal indices of the $\ell$-ification and the ones of the matrix polynomial are related by simple rules.
This property implies that the complete eigenstructure of the polynomial can be recovered even in the singular case.
Third, the eigenvectors and minimal bases of the matrix polynomial can be recovered from those of the $\ell$-ification.
Four, the $\ell$-ification presents one-sided factorizations, useful for performing conditioning and local backward error analyses.

In the particular case when $\ell$ divides $d$, we have introduced the family of block Kronecker matrix polynomials, which is a subfamily of strong block minimal bases pencils.
This family has allowed us to construct many companion $\ell$-ifications that are different from the Frobenius-like companion $\ell$-ifications in \cite{IndexSum}.
Furthermore, for any strong $\ell$-ification in the block Kronecker matrix polynomials family, we have shown that the eigenvectors and minimal bases of the matrix polynomial can be recovered from those of the $\ell$-ification without any extra computational cost.

\end{document}